\documentclass[12pt,english]{amsart}
                                                        
\usepackage[paperwidth=216mm,paperheight=279mm,inner=3.5cm,outer=3.5cm,top=2.5cm,bottom=2.5cm]{geometry}

\usepackage[T1]{fontenc}
\usepackage{amssymb}

\usepackage{colordvi}
\usepackage{graphicx}

\usepackage{paralist}
\usepackage{babel}
\usepackage{amsmath}
\usepackage{amsthm}
\usepackage{tikz}

\newcommand\mbb{\mathbb}

\newcommand\mcal{\mathcal}

\newcommand\ol{\overline}

\newcommand\wt{\widetilde}

\newcommand\sQ{\mcal{Q}}

\newcommand\sS{\mcal{S}}
\newcommand\sT{\mcal{T}}

\newcommand\sV{\mcal{V}}

\newcommand\C{\mbb{C}}

\renewcommand\P{\mbb{P}}
\newcommand\Q{\mbb{Q}}
\newcommand\R{\mbb{R}}
\newcommand\Z{\mbb{Z}}

\newcommand\lamda{\lambda}

\DeclareMathOperator*\conics{conics}

\DeclareMathOperator*\GL{GL}

\DeclareMathOperator*\rank{rank}

\newcommand\isom{\cong}

\renewcommand\epsilon{\varepsilon}

\renewcommand\phi{\varphi}

\renewcommand\theta{\vartheta}

\numberwithin{equation}{section}

\theoremstyle{plain}
\newtheorem{Thm}[equation]{Theorem}
\newtheorem{Prop}[equation]{Proposition}
\newtheorem{Cor}[equation]{Corollary}
\newtheorem{Lemma}[equation]{Lemma}

\newtheorem*{Thm*}{Theorem}
\newtheorem*{Prop*}{Proposition}
\newtheorem*{Cor*}{Corollary}
\newtheorem*{Lemma*}{Lemma}
\newtheorem*{Sublemma*}{Sublemma}
\newtheorem*{Conjecture*}{Conjecture}

\theoremstyle{definition}

\newtheorem{Example}[equation]{Example}

\newtheorem{Algorithm}[equation]{Algorithm}

\newtheorem{Remark}[equation]{Remark}

\newtheorem*{Constr*}{Construction}
\newtheorem*{Def*}{Definition}
\newtheorem*{Defs*}{Definitions}
\newtheorem*{Example*}{Example}
\newtheorem*{Examples*}{Examples}
\newtheorem*{Exercise*}{Exercise}
\newtheorem*{LemmaDef*}{Lemma and Definition}
\newtheorem*{Notation*}{Notation}
\newtheorem*{Problem*}{Problem}
\newtheorem*{Question*}{Question}
\newtheorem*{Remark*}{Remark}
\newtheorem*{Remarks*}{Remarks}
\newtheorem*{Warning*}{Warning}

\newtheorem*{Text*}{}

\begin{document}
\title{Quartic Curves and Their Bitangents}
\author{Daniel Plaumann}
\author{Bernd Sturmfels}
\author{Cynthia Vinzant}
\address{Dept.~of Mathematics, University of
  California, Berkeley, CA 94720, USA}
  \subjclass[2010]{Primary: 14H45;  Secondary: 13P15, 14H50, 14Q05}
%\date{\today}

\begin{abstract}
A smooth quartic curve in the complex projective plane has
$36$ inequivalent representations as a symmetric determinant of linear forms
and $63$ representations as a sum of three squares.
These correspond  to Cayley octads and Steiner complexes respectively.
We present exact algorithms for computing these objects
from the $28$ bitangents. This expresses
Vinnikov quartics as spectrahedra and
positive quartics as Gram matrices.
We explore the geometry of Gram spectrahedra and
we find equations for the variety of  Cayley octads.
Interwoven is an exposition of much of
the $19$th century theory of plane quartics.
\end{abstract}

\maketitle

\section{Introduction}
We consider smooth curves in the projective plane defined by ternary quartics
\begin{equation}
\label{intro:quartic}
 f(x,y,z) \,\,\, = \,\,\,  c_{400} x^4 + c_{310} x^3 y + c_{301} x^3 z + 
c_{220} x^2 y^2 + c_{211} x^2 y z + \cdots + c_{004} z^4 , 
\end{equation}
whose $15$ coefficients $c_{ijk}$ are parameters over the field $\Q$ of rational numbers.
Our goal is to devise exact algorithms for computing the two alternate representations
\begin{equation}
\label{intro:LMR}  f(x,y,z)\quad = \quad {\rm det}\bigl( x A + y B + z C \bigr), \qquad \qquad \qquad
\end{equation}
where $A,B,C$ are symmetric $4 \times 4$-matrices, and
\begin{equation}
\label{intro:SOS} \qquad
 f(x,y,z) \quad = \quad q_1(x,y,z)^2 \,+\,
 q_2(x,y,z)^2 \,+\, q_3(x,y,z)^2 ,
  \end{equation}
where the $q_i(x,y,z)$ are quadratic forms.
The representation (\ref{intro:LMR}) is of most interest
when the real curve $\mathcal{V}_\R(f)$ consists of two nested ovals.
Following Helton-Vinnikov \cite{MR2292953} and Henrion \cite{Hen},
one seeks real symmetric matrices $A,B,C$ whose span contains
a positive definite matrix. 
The representation (\ref{intro:SOS}) is of most interest
when the real curve $\mathcal{V}_\R(f)$ is empty.
Following Hilbert \cite{MR1510517} and
Powers-Reznick-Scheiderer-Sottile \cite{PRSS}, one seeks quadrics
 $q_i(x,y,z)$ with real coefficients.
We shall explain how to compute  
{\em all} representations (\ref{intro:LMR}) and (\ref{intro:SOS}) over~$\C$.

The theory of plane quartic curves is a delightful chapter of
$19$th century mathematics, with contributions by
Aronhold, Cayley, Frobenius, Hesse, Klein, Schottky, Steiner, Sturm
and many others. Textbook references include
\cite{Dol, MR0123600, MR0115124}.  It started in 1834 with
 Pl\"ucker's result \cite{012.0444cj} that
the complex curve $\mathcal{V}_\C(f)$ 
has $28$ bitangents. The linear form $ \,\ell =  \alpha x + \beta y + \gamma z \,$
of a bitangent satisfies the identity
$$ f(x,y,z) \quad = \quad \, g(x,y,z)^2  \,+\,\ell(x,y,z) \cdot h(x,y,z)  $$
for some quadric $g$ and some cubic $h$.
This translates into a system of polynomial equations in
$(\alpha: \beta: \gamma)$, and our algorithms 
start out by solving these equations.

Let $K$ denote the corresponding splitting field, that is,
the smallest field extension of $\Q$ that
contains the coefficients $\alpha, \beta, \gamma$ for all
$28$ bitangents.
The Galois group $ {\rm Gal}(K,\Q)$ is very far from being the
symmetric group $S_{28}$. In fact, if the coefficients
$c_{ijk}$ are general enough, it is the Weyl group of $E_7$ modulo its center,
\begin{equation}
\label{intro:galois}  \,\,
 {\rm Gal}(K,\Q) \, \isom \, W(E_7)/\{\pm 1\}  \, \isom \,
 {\rm Sp}_6(\Z/2\Z) . \end{equation}
This group has order $\, 8 ! \cdot 36 \, = 1451520 $,
and it is not solvable \cite[page 18]{Har}.
We will see a combinatorial representation of this Galois group 
 in Section~\ref{sec:octad} (Remark~\ref{rem:HesseTable}). It is based on
\cite[\S 19]{MR0123600} and
 \cite[Thm.~9]{DO}. The connection with
 ${\rm Sp}_6(\Z/2\Z)$ arises from the theory of
theta functions  \cite[\S 5]{Dol}. For further information see
\cite[\S II.4]{Har}.

Naturally, the field extensions needed for
(\ref{intro:LMR}) and (\ref{intro:SOS}) are much smaller
for special quartics. As our running
example we take the smooth quartic given by
$$ E(x,y,z) \quad = \quad 
25 \cdot (x^4+y^4+z^4)\,-\, 34 \cdot (x^2 y^2+x^2 z^2+y^2 z^2). $$
We call this the {\em Edge quartic}. It is one of the curves in the family
studied by William L.~Edge in \cite[\S 14]{Edge}, and it admits
a matrix representation  (\ref{intro:LMR}) over $\Q$:
\begin{equation}
\label{LMRforedge}
 E(x,y,z) \quad = \quad {\rm det}
               \begin{pmatrix}   0  &     x + 2 y  &   2 x + z &   y - 2 z  \\
                   x + 2 y   &    0  &      y + 2 z  &  -2 x + z \\
                    2 x + z  &  y + 2 z   &     0   &    x - 2 y  \\
                   \, y - 2 z  &  -2 x + z  &  x - 2 y  &     0    \end{pmatrix}. 
 \end{equation}
 The sum of three squares representation (\ref{intro:SOS}) is derived from the expression
\begin{equation}
\label{eq:Gram} 
 \begin{pmatrix} x^2 \! &\! y^2 \! & \! z^2 \! &\! xy \! 
& \! xz \! &\! yz \end{pmatrix}
\!\! 
\begin{pmatrix}
25 & -55/2 & -55/2 & 0 & 0 & 21 \\
-55/2 & 25 & 25 & 0 & 0 & 0 \\
-55/2 & 25 & 25 & 0 & 0 & 0 \\
0 & 0 & 0 & 21 & -21 & 0 \\
0 & 0 & 0 & -21 & 21 & 0 \\
21 & 0 & 0 & 0 & 0 & -84
\end{pmatrix} \!\!                
\begin{pmatrix}
 x^2 \\ y^2 \\ z^2 \\ xy \\ xz \\ yz \end{pmatrix} 
\end{equation}
 by factoring the above rank-$3$ matrix
  as $H^T \cdot H $ where $H$ is a complex $3 {\times} 6$-matrix.
The real quartic curve $\mathcal{V}_\R(E)$ consists of four ovals and is shown in Figure~\ref{fig:edge}.

\begin{center}
\begin{figure}
 \includegraphics[width=7.0cm]{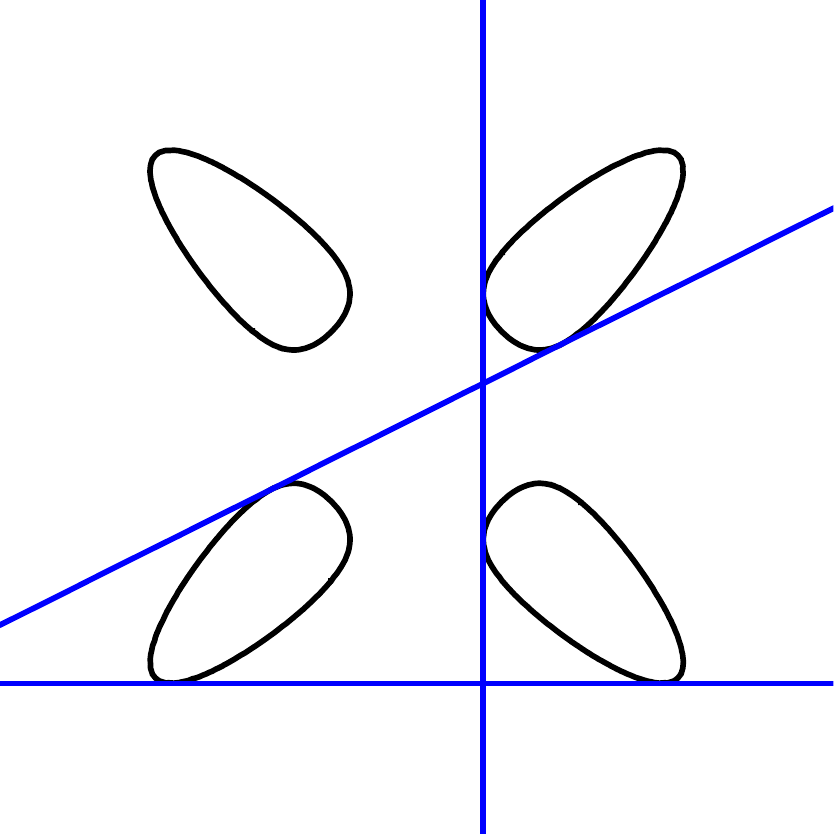} \quad 
 \includegraphics[width=6.8cm]{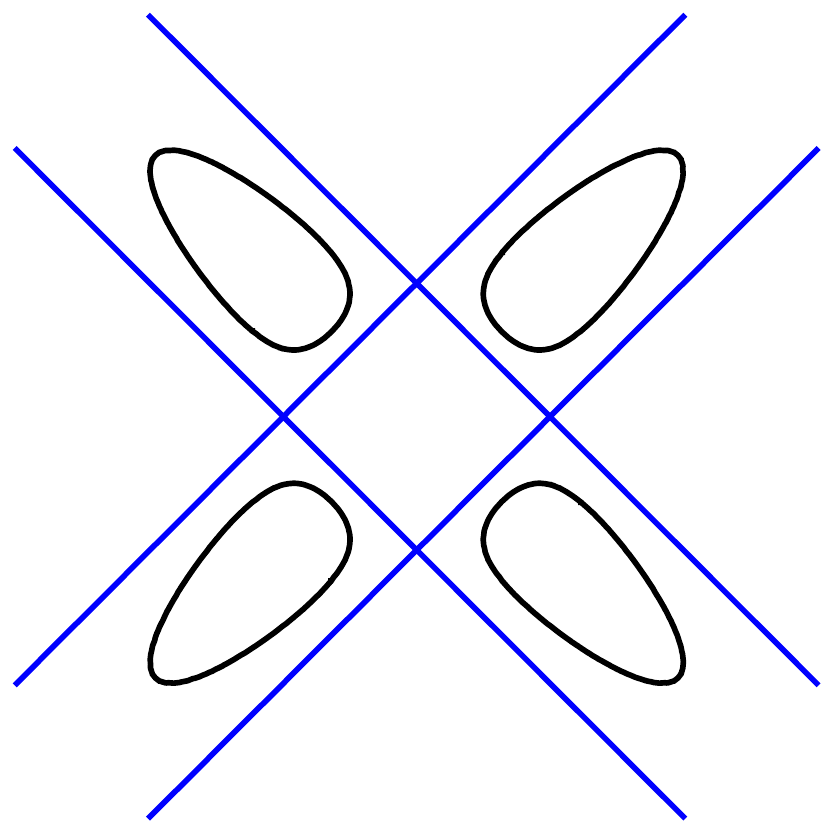}
\vskip -0.4cm
\caption{The Edge quartic and some of its $28$ bitangents}
\label{fig:edge}
\end{figure}
\end{center}

Each of the $28$ bitangents of the Edge quartic is defined over $\Q$,
but the four shown on the right in  Figure~\ref{fig:edge} 
are tangent at complex points of the curve.
The following theorem and Table~\ref{table:sixtypes} summarize the
possible shapes of real quartics.

\begin{Thm} \label{thm:sixtypes}
There are six possible topological types for a smooth
quartic curve $\mathcal{V}_\R(f)$ in the real projective plane.
They are listed in the first column of Table \ref{table:sixtypes}.
Each of these six types corresponds to only one connected
component in the complement of the discriminant
$\Delta$ in the  $14$-dimensional projective space of quartics.
\end{Thm}
\begin{table}[htb] \label{table:sixtypes}
\begin{tabular}{llrr}
The real curve & Cayley octad &  Real bitangents & Real Steiner complexes \\
\hline
4 ovals & 8 real points & 28 & 63  \\
3 ovals & 6 real points & 16 & 31 \\
2 non-nested ovals & 4 real points & 8 & 15  \\
1 oval & 2 real points & 4 & 7  \\
2 nested ovals & 0 real points & 4 & 15 \\
empty curve & 0 real points & 4 & 15 \\
\end{tabular}
\medskip
\caption{The six types of smooth quartics in the real projective plane.}% with respect to a real LMR
\end{table}

The classification result in Theorem \ref{thm:sixtypes} is due to Zeuthen \cite{Zeu}.
An excellent exposition can be found in 
Salmon's book \cite[Chapter VI]{MR0115124}.
Klein \cite[\S 5]{Kle} proved that each type is connected in the complement of the discriminant
$\{\Delta = 0\}$.
We note that $\Delta$ is a homogeneous polynomial of degree $27$
in the $15$ coefficients $c_{ijk}$ of $f$. As a preprocessing step in our
algorithms, we use the explicit formula for $\Delta$ given in \cite[Proposition 6.5]{SSS}
to verify that a given quartic curve $\mathcal{V}_\C(f)$ is smooth.

\medskip
The present paper is organized as follows.
In Section~2 we present an algorithm, based on
Dixon's approach \cite{33.0140.04}, for computing one
determinantal representation (\ref{intro:LMR}).
The resulting $4 {\times} 4$-matrices $A,B$ and $C$ specify
three quadratic surfaces in $\P^3$ whose intersection
consists of eight points, known as a {\em Cayley octad}.

In Section~3 we use Cayley octads to compute representatives for all
$36$ inequivalent classes of determinantal representations
(\ref{intro:LMR}) of the given quartic $f$. This is accomplished by a
combinatorial algorithm developed by Hesse in \cite{049.1317cj}, which realizes
the Cremona action \cite{DO} on the Cayley octads.  The output
consists of $36$ symmetric $8 {\times} 8$-matrices
(\ref{BitangentMatrix}).  These have rank $4$ and their $28$ entries
are linear forms defining the bitangents.

In Section~4 we focus on {\em Vinnikov quartics},
that is, real quartics consisting of two nested ovals.
Helton and Vinnikov \cite{MR2292953} proved the existence of
a representation (\ref{intro:LMR}) over $\R$. We present a symbolic
algorithm for computing that representation in practice.
Our method uses exact arithmetic and writes the convex inner oval explicitly
as a spectrahedron. This settles a question raised by Henrion \cite[\S 1.2]{Hen}.

In Section 5 we identify sums of three squares with {\em Steiner
  complexes} of bitangents, and we compute all $63$ Gram matrices,
i.e.~all $6 {\times} 6$-matrices of rank $3$ as in (\ref{eq:Gram}),
again using only rational arithmetic over $K$. This ties in with the
results of Powers, Reznick, Scheiderer and Sottile in \cite{PRSS},
where it was proved that a smooth quartic $f$ has precisely $63$
inequivalent representations as a sum of three squares
(\ref{intro:SOS}). They strengthened Hilbert's theorem in
\cite{MR1510517} by showing that precisely eight of these $63$ are
real when $f$ is positive.

Section~\ref{sec:Gram} is devoted to the boundary and facial structure
of the {\em Gram spectrahedron}. This is the six-dimensional
spectrahedron consisting of all sums of squares representations of a
fixed positive ternary quartic $f$. We show that its eight special
vertices are connected by $12$ edges that form two complete graphs
$K_4$.  We also study the structure of the associated semidefinite
programming problems.

Section~\ref{sec:octad2} is devoted to the variety of Cayley octads \cite[\S IX.3]{DO}.
We discuss its defining equations and its boundary strata, we compute the
discriminants of (\ref{intro:LMR}) and (\ref{intro:SOS}), and
we end with a classification of nets of real quadrics in $\P^3$.

We have implemented most of the algorithms 
presented in this paper in the system
{\tt SAGE}\footnote{\tt www.sagemath.org}. Our software and
supplementary material on quartic curves and Cayley octads
can be found at 
 {\tt math.berkeley.edu/$\sim$cvinzant/quartics.html}.

\section{Computing a Symmetric Determinantal Representation}\label{sec:LMR}

We now prove, by way of a constructive algorithm, that
every smooth quartic admits a symmetric determinantal
representation (\ref{intro:LMR}). First we compute the
$28$ bitangents, $\ell =  \alpha x + \beta y + \gamma z \,$. Working on the affine chart
$\{\gamma = 1\}$, we equate
$$ f \bigl(x,y,- \alpha x - \beta y \bigr) \,\, = \,\,
(\kappa_0 x^2 + \kappa_1 x y + \kappa_2 y^2 )^2, $$
eliminate $\kappa_0, \kappa_1, \kappa_2$, 
and solve the resulting system for the unknowns
$\alpha$ and $\beta$. This constructs the splitting field $K$
for the given $f$ as a finite extension of~$\Q$.
All further computations in this section are performed
via rational arithmetic in $K$.

Next consider any one of the $\binom{28}{3} = 3276$ triples of bitangents.
Multiply their defining linear forms. The resulting polynomial
$\,v_{00} = \ell_1 \ell_2 \ell_3 \,$ is a {\em contact cubic} for $\sV_\C(f)$, which means that 
the ideal $\langle v_{00}, f \rangle $ in $K[x,y,z]$ 
defines six points in $\P^2$ each of multiplicity 2.
Six points that span three lines in $\P^2$ impose independent conditions on cubics, so
the space of cubics in the radical of $\langle v_{00},f \rangle $
is  $4$-dimensional over $K$. We extend $\{v_{00}\}$ to 
a basis $\{v_{00}, v_{01}, v_{02} , v_{03} \}$ of that space.

Max Noether's Fundamental Theorem  \cite[\S~5.5]{MR1042981} can be applied
to the  cubic $v_{00}$ and the quartic $f$. It implies that
a homogeneous polynomial lies in $\langle v_{00}, f \rangle$ if 
it vanishes to order two at each of the six points of $\sV_\C \bigl(\langle v_{00},f \rangle \bigr)$.
The latter property holds for the sextic forms $v_{0i} v_{0j}$. Hence 
$v_{0i} v_{0j}$ lies in $\langle v_{00}, f \rangle$ for  $1 \leq i \leq j \leq 3$.
Using the Extended Buchberger Algorithm, we can compute
cubics $v_{ij}$ such that
\begin{equation}
\label{twobytwominors}
 v_{0i} v_{0j} - v_{00} v_{ij} \,\in \, \langle  f \rangle. 
 \end{equation}
 
We now form a symmetric $4 {\times} 4$-matrix $V$
whose entries are cubics in $K[x,y,z]$:
$$ V \quad = \quad
\begin{pmatrix}
v_{00} & v_{01} & v_{02} & v_{03} \\
v_{01} & v_{11} & v_{12} & v_{13} \\
v_{02} & v_{12} & v_{22} & v_{23} \\
v_{03} & v_{13} & v_{23} & v_{33} 
\end{pmatrix} . $$
The following result is due to Dixon \cite{33.0140.04},
and it almost solves our problem.

\begin{Prop} \label{prop:dixon}
Each entry of the adjoint $V^{\rm adj}$
is a linear form times $f^2$, and 
$$ {\rm det}( f^{-2} \cdot V^{\rm adj}) \quad = \quad \gamma \cdot f(x,y,z) \qquad
\hbox{for some constant $ \gamma \in K$.} $$ 
Hence, if $\,{\rm det}(V) \not = 0\,$ then 
$\,f^{-2} \cdot V^{\rm adj}\,$ gives a linear matrix
representation (\ref{intro:LMR}).
\end{Prop}

\begin{proof}
Since $v_{00} \not\in \langle f \rangle$, the condition 
(\ref{twobytwominors}) implies that,
over the quotient ring $K[x,y,z]/\langle f \rangle$,
the matrix $V$ has rank $1$.
 Hence, in the polynomial ring  $K[x,y,z]$,  the cubic $f$ divides
 all $2\times 2$ minors of $V$. This implies that $f^2$
divides all $3\times 3$ minors of $V$, and $f^3$ divides $\det(V)$.
As the entries of $V^{\rm adj}$ have degree $9$, it follows that $V^{\rm
  adj} = f^2 \cdot W$, where $W$ is a symmetric matrix whose entries are linear forms.
Similarly, as $\det(V)$ has degree $12$, we have $\det(V) = \delta f^3$ for some $\delta
\in K$, and $\delta \not= 0$ unless ${\rm det}(V)$ is identically zero.
Let $I_4$ denote the identity matrix. Then
$$ \delta f^3 \cdot I_4 \,\, = \,\, \det(V)\cdot I_4 \,\, = \,\,
V \cdot V^{\rm adj}  \,\, = \,\,f^2 \cdot V \cdot W .$$
Dividing by $f^2$ and taking determinants yields 
$$ \delta^4 f^4 \,\, = \,\, \det(V) \cdot \det(W) \,\, = \,\,
\delta f^3 \cdot {\rm det}(W). $$
This implies the desired identity  $\,\det(W)=\delta^3 f$. 
\end{proof}

We now identify the conditions to ensure that
${\rm det}(V)$ is not the zero polynomial.

\begin{Thm} \label{thm:1260}
The determinant of $V$ vanishes if and only if
the six points of $\mathcal{V}_\C(f,\ell_1\ell_2\ell_3)$,
 at which the bitangents 
$\ell_1, \ell_2,\ell_3$ touch the quartic curve
$\mathcal{V}_\C(f)$, lie on a conic in $\P^2$.
This happens for precisely 
$1260$ of the $3276$ triples of bitangents.
\end{Thm}

\begin{proof}
Dixon \cite{33.0140.04} proves the first assertion. The census of triples
appears in the table on page 233 in Salmon's book \cite[\S 262]{MR0115124}.
It  is best understood via the Cayley octads in Section~3.
For further information see Dolgachev's notes \cite[\S 6.1]{Dol}. 
\end{proof}

\begin{Remark}
Let $\ell_1,\ell_2,\ell_3$ be any three bitangents of $\mathcal{V}_\C(f)$.
If the six intersection points with 
$\mathcal{V}_\C(f)$  lie on a conic,
the triple $\{\ell_1,\ell_2,\ell_3\}$ is called \emph{syzygetic}, otherwise \emph{azygetic}. 
A smooth quartic $f$ has  $1260 $ syzygetic and $2016$
azygetic triples of bitangents.
Similarly, a quadruple $\{\ell_1,\ell_2,\ell_3,\ell_4\}$ of
bitangents is called \emph{syzygetic} if its eight contact points lie
on a conic and \emph{azygetic} if they do not. Every syzygetic triple
$\ell_1,\ell_2,\ell_3$ determines a fourth bitangent $\ell_4$ with which
it forms a syzygetic quadruple. Indeed, if the contact points of
$\ell_1,\ell_2,\ell_3$ lie on a conic with defining polynomial $q$,
then $q^2$ lies in the ideal $\langle f, \ell_1 \ell_2 \ell_3 \rangle$,
so that $q^2 = \gamma f + \ell_1 \ell_2 \ell_3 \ell_4$, and 
the other two points in $\mathcal{V}_\C(f,q)$ must
be the contact points of the bitangent $\ell_4$.
\end{Remark}

\begin{Algorithm} \label{alginsec2}
Given a smooth ternary quartic $f\in\Q[x,y,z]$, we compute the
splitting field $K$ over which the $28$ bitangents of $\sV_\C(f)$ are
defined. We pick a random triple of bitangents and construct the
matrix $V$ via the above method. If $\det(V)\neq 0$,
we compute the adjoint of $V$ and divide by $f^2$, obtaining the
desired determinantal representation of $f$ over $K$. If $\det(V)=0$,
we pick a different triple of bitangents. On each iteration, the
probability for $\det(V)\neq 0$ is $\frac{2016}{3276}=\frac{8}{13}$.
% To compute a symmetric
% determinantal representation of $f$, we
% pick a random triple of bitangents
% and apply the above method to find the corresponding matrix $V$.
% With probability $\frac{2016}{3276}=\frac{8}{13}$ we have
% ${\rm det}(V) \not= 0$ and then we compute
%  the adjoint of $V$ and divide by $f^2$.
%  If ${\rm det}(V) = 0$ then we repeat.
 \end{Algorithm}
 
\begin{Example}
The diagram on the left of Figure \ref{fig:edge} shows an 
azygetic triple of bitangents to the Edge quartic. Here, the six points of tangency
do not lie on a conic. The representation of the Edge quartic in
(\ref{LMRforedge}) is produced by Algorithm~\ref{alginsec2} starting from 
 the contact cubic $\,v_{00} =2 (y + 2 z)(-2 x + z) (x - 2 y)$. \qed
\end{Example}

\section{Cayley Octads and the Cremona Action} \label{sec:octad}

Algorithm \ref{alginsec2} outputs a 
matrix $M=x A +y B+z C$ where $A,B,C$ are symmetric $4 {\times} 4$-matrices 
with entries in the subfield $K$ of $\C$ over which all $28$ bitangents of $\sV_\C(f)$ are defined.  Given one such 
representation (\ref{intro:LMR})
  of the quartic $f$, we shall construct a representative from each of the 35 other equivalence classes.
 Two representations (\ref{intro:LMR}) are considered {\em
  equivalent} if they are in the same orbit under the action of ${\rm
  GL_4}(\C)$ by conjugation $M \mapsto U^T M U$.
We shall present an algorithm for the following result.
It performs rational arithmetic over the splitting field
  $K$   of the $28$ bitangents, and it
    constructs one representative for each of the $36$ orbits. 
    
 \begin{Thm}[Hesse \cite{049.1317cj}]
 \label{thm:36}
Every smooth quartic curve $f$ has exactly $36$ equivalence classes
of linear symmetric determinantal representations (\ref{intro:LMR}).
\end{Thm}

 Our algorithms begins by
      intersecting the three quadric surfaces seen in   $M$:
 \begin{equation}
 \label{threequadrics}   u A u^T \, = \, u B u^T \, = \, u C u^T \,\, = \,\, 0 \quad \hbox{where} \,\,\,\,
  u = (u_0:u_1:u_2:u_3) \,\, \in \, \P^3(\C) .
  \end{equation}
  These equations have eight solutions $O_1,\ldots,O_8$.
This is the {\em Cayley octad} of $M$. In general, a Cayley octad
is the complete intersection of three quadrics in $\P^3(\C)$.

The next proposition gives a bijection between
the $28$ bitangents of $\mathcal{V}_\C(f)$  and
the lines $\ol{O_iO_j}$ for $1 \leq i \leq j \leq 8$.
The combinatorial structure 
of this configuration of $28$ lines in $\P^3$ plays an important role for our algorithms.

\begin{Prop}\label{Prop:OctadToBitangents}
  Let $O_1,\ldots,O_8$ be the Cayley octad defined above. Then the 
  $28$ linear forms $O_i M O_j^T \in \C[x,y,z]$ are the equations of
  the   bitangents of $\mathcal{V}_\C(f)$.
\end{Prop}

\begin{proof}
  Fix $i\neq j$.  After a change of basis on $\C^4$ given by a matrix
  $U\in{\rm GL}_4(\C)$ and replacing $M$ by $U^TMU$, we may assume
  that $O_i=(1,0,0,0)$ and $O_j=(0,1,0,0)$. The linear form
  $b_{ij}=O_i M O_j^T$ now appears in the matrix:
  \[
M=\left(
  \begin{array}{c|c}
    \begin{matrix}
      0 & b_{ij}\\
      b_{ij} & 0
    \end{matrix} & M'\\
    \hline
    (M')^T & \ast    
  \end{array}
\right).
\]
Expanding $\det(M)$ and sorting for terms containing $b_{ij}$ shows
that $f{=}\det(M)$ is congruent to ${\rm det}(M')^2$ modulo $\langle
b_{ij} \rangle$. This means that $b_{ij}$ is a bitangent.
\end{proof}

Let $O$ be the $8 \times 4$-matrix with rows given by the Cayley
octad.  The symmetric $8 {\times} 8$-matrix $O M O^T$ has rank $4$,
and we call it the {\em bitangent matrix} of $M$.  By the definition
of $O$, the bitangent matrix has zeros on the diagonal, and,
by Proposition \ref{Prop:OctadToBitangents}, its $28$ off-diagonal
entries are precisely the equations of the bitangents:
\begin{equation}
\label{BitangentMatrix}
 O M O^T \quad = \quad
 \begin{smaller}
\begin{pmatrix} 
0 & b_{12} & b_{13} & b_{14} & b_{15} & b_{16} & b_{17} & b_{18} \\
b_{12} & 0 & b_{23} & b_{24} & b_{25} & b_{26} & b_{27} & b_{28} \\
b_{13} & b_{23} &  0 & b_{34} & b_{35} & b_{36} & b_{37} & b_{38} \\
b_{14} & b_{24} & b_{34} & 0 & b_{45} &  b_{46} & b_{47} & b_{48} \\
b_{15} & b_{25} & b_{35} & b_{45} & 0 &  b_{56} & b_{57} & b_{58} \\
b_{16} & b_{26} & b_{36} & b_{46} &  b_{56} & 0 & b_{67} & b_{68} \\
b_{17} & b_{27} & b_{37} & b_{47} &  b_{57} & b_{67} & 0 & b_{78} \\
b_{18} & b_{28} & b_{38} & b_{48} &  b_{58} & b_{68} & b_{78} & 0 
\end{pmatrix}.
\end{smaller}
\end{equation}

\begin{Remark} We can see that the octad $O_1, \hdots, O_8$ consists of $K$-rational points of $\P^3$: 
To see this, let $K'$ be the field of definition of the octad over $K$. Then any element $\sigma$ of ${\rm Gal}(K'\!:\!K)$ acts on the octad by permutation, and thus permutes the indices of the bitangents, $b_{ij}$. On the other hand, as all bitangents are defined over $K$, $\sigma$ must fix $b_{ij}$ (up to a constant factor).  Thus the permutation induced by $\sigma$ on the octad must be the identity and ${\rm Gal}(K'\!:\!K)$ is the trivial group.

\end{Remark}
\begin{Example}
The symmetric matrix $M$ in (\ref{LMRforedge}) determines the 
Cayley octad
$$
O^T \quad = \quad \begin{pmatrix}
 \,1&0&0&0&-1&1&1&3\\
\,0&1&0&0&3&-1&1&1\\
\,0&0&1&0&1&3&1&-1 \\
\,0&0&0&1&-1&-1&3&-1 \end{pmatrix}. $$
All the $28$ bitangents of $E(x,y,z)$ are revealed 
in the bitangent matrix $\,O M O^T = $ 
\[ 
\begin{tiny}
\!\!\!
\begin{pmatrix} 
0\!\! & \!x+2y\!\! & \!2x+z\!\! & \!y-2z\!\! & \!5x \! + \! 5y \! + \! 3z
 \!\! & \! 5x \! - \! 3y \! + \! 5z\!\! & \! 3x \!+ \! 5y\! - \! 5z\!\! & \!-x \!+\!y \!+\!z\\
 x+2y\!\! & \!0\!\! & \!y+2 z\!\! & \!-2x+z\!\! & \!x\!- \! y \! + \! z\!\! &
 \!3 x\! + \! 5y\! + \!
 5z\!\! &\! \!-5x\!+\!3y\!+\! 5z\!\! & \!5x\!+\!5y\!-\!3z\\
 2x+z\!\! & \!y+2z\!\! & \!0\!\! & \!x-2y\!\! & \!\!-3x \!+\!5z\!+\!5y\!\! &
 \!x\!-\!z\!+\!y\!\! & \!5x\!+\!3z\!-\!5y\!\! & \!5x\!+\!5z\!+\!3y\\
 y-2z\!\! & \!-2x+z\!\! & \!x-2y\!\! & \!0\!\! & \!\!  -3y\!+\!5z\!-\!5x\!\! &
 \!-5y\!-\!3z\!+\!5x\!\! & \!-y\!-\!z\!-x\!\! & \!5y\!-\!5z\!-\!3x \\
\!  5x\!+\!5y\!+\!3z\!\! & \!x\!-\!y\!+\!z\!\! & \!\!-3x\!+\!5z\!+\!5y\!\! &
 \!-3y\!+\!5z\!-\!5x\!\!
 & \!0\!\! & \!24y+12z\!\! & \!-12x+24z\!\! & \!24x+12y\\ \!
 5x\!-\!3y\!+\!5z\!\! & \!3x\!+\!5y\!+\!5z\!\! & \!x\!-\!z\!+\!y\!\! &
 \!-5y\!-\!3z\!+\!5x\!\! & \!24y+12z\!\! & \!0\!\! & \!24x-12y\!\! & \!
 12x+24z\\ \! 3x\!+\!5y\!-\!5z\!\! & \!\!-5x\!+\!3y\!+\!5z\!\!  &
 \!5x\!+\!3z\!-\!5y\!\! & \!-y\!-\!z\!-\!x\!\! & \!-12x+24z\!\! &
 \!24x-12y\!\! & \!0\!\! & \!24y-12z \\ -x\!+\!y\!+\!z\!\! &
 \!5x\!+\!5y\!-\!3z\!\! & \!5x\!+\!5z\!+\!3y\!\! & \!5y\!-\!5z\!-\!  3x\!\!
 & \!24x+12y\!\! & \!12x+24z\!\! & \!24y-12z\!\! & \!0
\end{pmatrix}
\end{tiny} 
\]
Each principal $4 {\times} 4$-minors of this matrix is a multiple of
$E(x,y,z)$, as in (\ref{eq:symmetricdeterminant}).  \qed
\end{Example}

Each principal $3 {\times} 3$-minor of the bitangent matrix (\ref{BitangentMatrix})
is a contact cubic $2 b_{ij} b_{ik} b_{jk}$ of $\mathcal{V}_\C(f)$
and can serve as the starting point for the procedure in Section 2.
Hence, each principal $4 {\times} 4$-minor  $M_{ijkl}$ of (\ref{BitangentMatrix})
represents the same quartic:
\begin{equation}\label{eq:symmetricdeterminant} \begin{matrix}
\det(M_{ijkl}) \,\,\,= \,\,\,  \hbox{a non-zero scalar multiple of} \,\, f(x,y,z) 
\qquad \qquad \qquad \qquad
\\ \qquad = \,\,\,
b_{i j}^2b_{ k l}^2+b_{ i k}^2b_{ j l}^2+b_{ i l}^2b_{ j k}^2-2(b_{ i j}b_{ i k}b_{ j l}b_{ k l}+b_{ i j}b_{ i l}b_{ j k}b_{ k l}+b_{ i k}b_{ i l}b_{ j k}b_{ j l}) .
\end{matrix}
\end{equation}
However, all these $\binom{8}{4} = 70$ realizations of (\ref{intro:LMR}) lie in the
same equivalence class. 

In what follows,
we present a simple recipe due to Hesse \cite{049.1317cj} for finding $35$ alternate bitangent matrices, 
each of which lies in a different  ${\rm GL}_4(\C)$-orbit. 
This furnishes all $36$ inequivalent determinantal representations promised
in Theorem \ref{thm:36}.
We begin with a remark that explains the number $1260$ in Theorem~\ref{thm:1260}.

\begin{Remark}\label{rem:TrianglesAndSquares}
We can use the combinatorics of the Cayley octad to classify syzygetic collections of 
bitangents. There are $56 $ triples $\bigtriangleup$ of the form
$\{b_{ij},b_{ik},b_{jk}\}$. 
Any such triple is azygetic, by the if-direction in Theorem~\ref{thm:1260},
because the cubic $b_{ij} b_{ik} b_{jk}$
appears on the diagonal
of the adjoint of the invertible matrix $M_{ijkl}$. 
Every product of an azygetic triple of bitangents appears as a $3\times 3$ minor of 
exactly one of the 36 inequivalent bitangent matrices, giving $36 \cdot 56 = 2016$ azygetic triples of bitangents and $\binom{28}{3}-2016 = 1260$ syzygetic triples.  

A quadruple of bitangents of type $\square$ is of the form
$\{b_{ij},b_{jk},b_{kl},b_{il}\}$. Any such quadruple is syzygetic.
Indeed, equation \eqref{eq:symmetricdeterminant} implies 
$\,f + 4(b_{ij}b_{jk}b_{kl}b_{il}) \,=\,
(b_{ij}b_{kl}-b_{ik}b_{jl}+b_{il}b_{jk})^2$,
and this reveals a conic containing the eight points of contact.
\end{Remark}

Consider the following matrix which is gotten by
permuting the entries of $M_{ijkl}$:
\[
M'_{ijkl} \quad = \quad
\left(
  \begin{matrix}
    0       & b_{ kl} & b_{ jl} & b_{ j k}\\
    b_{ kl} & 0       & b_{ il} & b_{ i k}\\
    b_{ jl} & b_{ il} & 0       & b_{ i j}\\
    b_{ j k} & b_{ i k} & b_{ i j} & 0
  \end{matrix}
\right).
\]
This procedure does not change the determinant:
${\rm det}(M'_{ijkl}) =  {\rm det}(M_{ijkl}) = f$.
This gives us $70$ linear determinantal representations 
(\ref{intro:LMR}) of the quartic $f$, one for each 
quadruple $I = \{i,j,k,l\} \subset \{1,\ldots,8\}$.
These are equivalent in pairs:

\begin{Thm} \label{Thm:BifidSubstitutions}
If $I\neq J$
are quadruples in $\{1,\hdots, 8\}$, then the
symmetric matrices $M'_I$ and $M'_J$ are
in the same ${\rm GL}_4(\C)$-orbit if and only if
$I$ and $J$ are disjoint.
None of these orbits contains the original matrix 
$M = xA + yB + zC$.
\end{Thm}

\begin{proof}
Fix $I = \{1,2,3,4\}$ and note the following identity
in $K[x,y,z,u_0,u_1,u_2,u_3]$:
\[
\left(
  \begin{matrix}
    u_0\\ u_1\\ u_2\\ u_3
  \end{matrix}
\right)^{\!\! T} \!\!\!\!
\left(
  \begin{matrix}
    0       & b_{12} & b_{13} & b_{14}\\
    b_{12} & 0       & b_{23} & b_{24}\\
    b_{13} & b_{23} & 0       & b_{34}\\
    b_{14} & b_{24} & b_{34} & 0
  \end{matrix}
\right) \!\!
\left(
  \begin{matrix}
    u_0\\ u_1\\ u_2\\ u_3
  \end{matrix}
\right)
\,=\,
u_0u_1u_2u_3
\left(
  \begin{matrix}
    u_0^{-1}\\ u_1^{-1}\\ u_2^{-1}\\ u_3^{-1}
  \end{matrix}
\right)^{\!\! T}
\!\! \!\!
\left(
  \begin{matrix}
    0       & b_{34} & b_{24} & b_{23}\\
    b_{34} & 0       & b_{14} & b_{13}\\
    b_{24} & b_{14} & 0       & b_{12}\\
    b_{23} & b_{13} & b_{12} & 0
  \end{matrix}
\right) \!\!
\left(
  \begin{matrix}
    u_0^{-1}\\ u_1^{-1}\\ u_2^{-1}\\ u_3^{-1}
  \end{matrix}
\right)
\]
This shows that the Cayley octad of $M'_{1234}$
is obtained from the Cayley octad of $M_{1234}$
by applying the  {\em Cremona transformation} at 
$O_1, O_2, O_3, O_4$. 
Equivalently, observe that the standard basis vectors of $\Q^4$ are the first four points in 
the Cayley octads of both $M_{1234}$ and $M_{1234}'$, and
if $O_i = (\alpha_i:\beta_i:\gamma_i:\delta_i)$
for $i=5,6,7,8$ belong to the Cayley octad of $M_{1234}$, then 
$O'_i = ({\alpha_i}^{-1}:\beta_i^{-1}:\gamma_i^{-1}:\delta_i^{-1})$
for $i=5,6,7,8$ belong to the Cayley octad $O'$ of $M'_{1234}$.

Thus the transformation from $M_{ijkl}$ to $M'_{ijkl}$
corresponds to the Cremona action ${\rm cr}_{3,8}$
on Cayley octads, as described on page 107
in the book of Dolgachev and Ortland \cite{DO}.
Each Cremona transformation changes the 
projective equivalence class of the Cayley octad,
and altogether we recover the $36$ distinct classes.
That $M'_I$ is equivalent to $M'_J$ when
$I$ and $J$ are disjoint can be explained by the following
result due to Coble \cite{Cob}. See \cite[\S III.3]{DO}
for a derivation in modern terms.
\end{proof}

\begin{Thm} \label{thm:coble}
Let $O$ be an unlabeled configuration of
eight points in linearly general position in $\P^3$.
Then $O$ is a Cayley octad
(i.e.~the intersection of three quadrics)
if and only if $O$ is self-associated (i.e.~fixed under Gale duality; cf.~\cite{EP}).
\end{Thm}

The Cremona action on Cayley octads was known classically 
as the {\em bifid substitution}, a term coined by Arthur Cayley himself.
We can regard this as a combinatorial rule that permutes and scales the
$28$ entries of the $8 {\times} 8$ bitangent matrix:

\begin{Cor} 
\label{cor:bifid}
The entries of the two bitangent matrices
$O M_{1234} O^T = (b_{ij})$ and $O' M'_{1234} O'^T = (b'_{ij})$
are related by non-zero scalars in the field $K$ as follows:
\[ \hbox{The linear form \ }
 b'_{ij} \text{\ \ is a scalar multiple of } \left\{
   \begin{array}{ll}
b_{kl} & \text{\rm if }\{i,j,k,l\}=\{1,2,3,4\}, \\
b_{ij} & \text{\rm if }|\{i,j\}\cap\{1,2,3,4\}|=1, \\
b_{kl} & \text{\rm if }\{i,j,k,l\}=\{5,6,7,8\}.
  \end{array}
\right.
\]
\end{Cor}

\begin{proof}
The first case is the definition of $M'_{1234}$.
For the second case we note that
\begin{equation}
\label{twoeqns}
\begin{matrix}
& b_{15} & = & O_1 M_{1234} O^T_5 & = & \beta_5b_{12}+\gamma_5b_{13}+\delta_5b_{14} \\
{\rm and}
& b'_{15}& = & O_1^{\prime}M'_{1234} O^{' T}_5 & = & \beta_5^{-1}b_{34}+\gamma_5^{-1}b_{24}+\delta_5^{-1}b_{23},
\end{matrix}
\end{equation}
by Proposition~\ref{Prop:OctadToBitangents}.
The identity $\,O_5 M_{1234} O_5^T= 0$, when combined with (\ref{twoeqns}),
translates into
$\,\alpha_5b_{15}+\beta_5\gamma_5\delta_5b'_{15}  = 0$, and hence
$\, b_{15}' = - \alpha_5 \beta_5^{-1} \gamma_5^{-1} \delta_5^{-1} b_{15}$.
For the last case we consider any pair $\{i,j\} \subset \{5,6,7,8\}$.
We know that  $b'_{ij} = \nu  b_{kl}$, for some $\nu \in K^*$ and 
$\{k,l\}\subset\{5,6,7,8\}$, by the previous two cases. 
We must exclude the possibility $\{k,l\} \cap \{i,j\} \neq \emptyset$.
After relabeling this would mean $b'_{56}  = \nu b_{56}$
or  $b'_{56} = \nu b_{57}$.
If  $b'_{56}  = \nu b_{56}$ then
the lines $\{b'_{12},b'_{25},b'_{56},b'_{16}\}$
and $\{b_{34},b_{25},b_{56},b_{16}\}$ coincide.
This is impossible because the left quadruple  is syzygetic while the right quadruple is not,
by Remark \ref{rem:TrianglesAndSquares}.
Likewise, $b'_{56} = \nu b_{57}$ would imply that
the  azygetic triple $\{b'_{15},b'_{56},b'_{16}\}$ corresponds to
  the syzygetic triple $\{b_{15},b_{57},b_{16}\}$. 
  \end{proof}

\begin{Remark}\label{rem:HesseTable}
The $35$ bifid substitutions of the Cayley octad are indexed by partitions 
 of $[8] = \{1,2,\ldots,8\}$ into pairs of $4$-sets. They are discussed 
in modern language in \cite[Prop.4, page 172]{DO}.
Each bifid substitution determines a
permutation of the set $ \binom{[8]}{2} = \bigl\{ \{i,j\} : 1{\leq} i {<} j {\leq} 8 \bigr\}$.
For instance,  the bifid partition $1234|5678$
determines the permutation 
in Corollary~\ref{cor:bifid}. Hesse   \cite[page 318]{049.1317cj} wrote these
$35$ permutations of $\binom{[8]}{2}$ explicitly in a table of format $35 {\times} 28$.
Hesse's remarkable table is a combinatorial 
realization of the Galois group (\ref{intro:galois}). Namely, 
$W(E_7)/\{\pm 1\}$ is the subgroup of column permutations that fixes 
the rows.
\end{Remark}

  We conclude this section with a remark on the real case. Suppose
  that $f$ is given by a real symmetric determinantal representation
(\ref{intro:LMR}),
  i.e.~$f=\det(M)$ where $M=xA+yB+zC$ and $A,B,C$ are real symmetric
  $4\times 4$-matrices. By \cite[\S 0]{MR1225986}, such a representation exists for every smooth real quartic $f$. 
  Then the quadrics $uAu^T, uBu^T, uCu^T \in
  K[u_0,u_1,u_2,u_3]_2$ defining the Cayley octad are real, so that
  the points $O_1,\dots,O_8$ are either real or come in conjugate
  pairs.

\begin{Cor}\label{cor:realoctad}
 Let $M=xA+yB+zC$ be a real symmetric matrix
  representation of $f$ with Cayley octad $O_1, \hdots, O_8$. Then the
  bitangent $O_i^TMO_j$ is defined over $\R$ if and only if $O_i$ and $O_j$ are
either  real or form a conjugate pair, $O_i = \overline{O_j}$. 
\end{Cor}

From the possible numbers of real octad points we can infer the 
numbers of real bitangents stated in Table~\ref{table:sixtypes}.
 If $2k$ of the eight points are real, then
there are $4-k$ complex conjugate pairs, giving $\binom{2k}{2} + 4-k =
2k^2-2k+4$ real bitangents. 
 
\section{Spectrahedral Representations of Vinnikov Quartics}

The symmetric 
determinantal representations $f=\det(M)$ of a ternary quartic
$f \in \Q[x,y,z]$ are grouped into $36$ 
 orbits under the action of $\GL_4(\C)$ given by $M\mapsto
T^TMT$. The algorithms in Sections \ref{sec:LMR} and \ref{sec:octad}
construct representatives for all $36$ orbits.
If we represent each orbit by its
$8 {\times} 8$-bitangent matrix  (\ref{BitangentMatrix}), 
then this serves
as a classifier for the $36$ orbits. Suppose 
we are given any other symmetric linear matrix representation
$M=xA+yB+zC$ of the same quartic $f$, and our
task is to identify in which of the $36$ orbits it lies. We do this
by computing the Cayley octad $O$ of $M$ and the resulting 
bitangent matrix $OM O^T$. That $8 {\times} 8$-matrix can be located in
our list of $36$ bitangent matrices by comparing
principal minors of size $3 {\times} 3$. These minors are
products of azygetic triples of bitangents, and they uniquely
identify the orbit since there are $2016 = 36 \cdot 56$   azygetic triples.

We now address the problem of finding matrices
$A, B$ and $C$ whose entries are real numbers. 
Theorem  \ref{thm:sixtypes} shows that this
 is not a trivial matter because
none of the $36$ bitangent matrices in
(\ref{BitangentMatrix}) has only real entries, unless the
 curve $\mathcal{V}_\R(f)$ consists of four ovals
(as in Figure \ref{fig:edge}).
We discuss the case
when the curve is a {\em Vinnikov quartic}, which means that
$\mathcal{V}_\R(f)$ consists of two nested ovals.

\begin{figure}
\includegraphics[width=7.4cm]{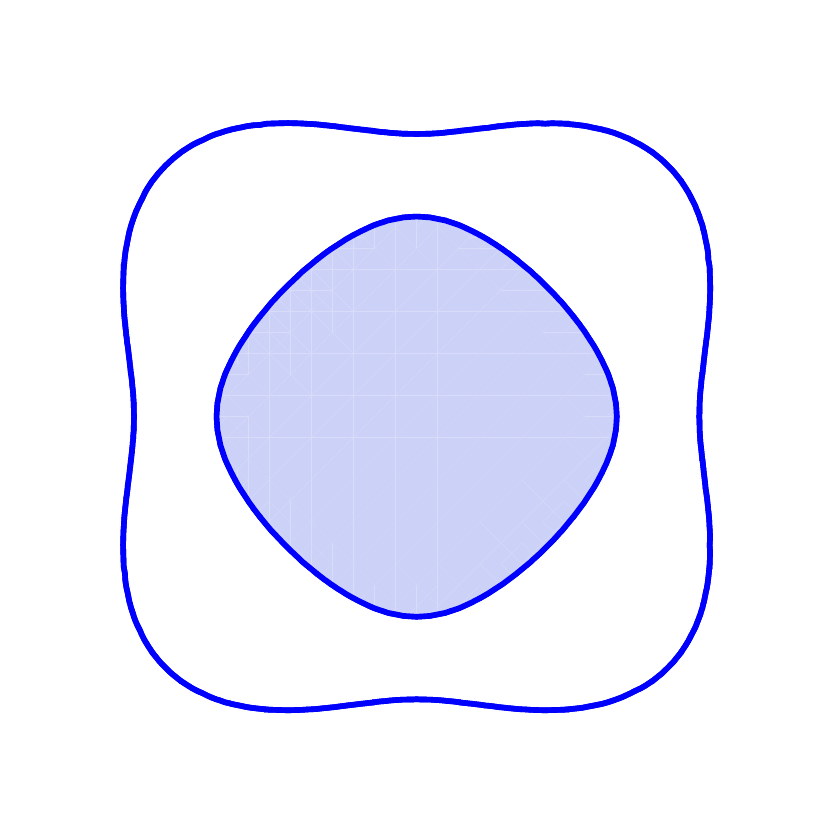} 
\vskip -0.9cm
\caption{The Vinnikov quartic in Example \ref{ex:vinnikov}.}
\label{fig:vinnikov}
\end{figure}

As shown in \cite{MR2292953}, the region bounded by the inner oval
corresponds exactly to 
$$ \bigl\{ (x,y,z) \in \R^3 \,\,:\,\,
x A \,+ \,y B \,+ \,z C \,\,\, \text{is positive definite} \bigr\} ,$$
a convex cone. This means that the inner oval is a {\em
  spectrahedron}.  The study of such {\em spectrahedral
  representations} is of considerable interest in convex optimization.
Recent work by Henrion \cite{Hen} underscores the difficulty of this
problem for curves of genus $g \geq 2$, and in the last two paragraphs
of \cite[\S 1.2]{Hen}, he asks for the development of a practical
implementation.  This section constitutes a definitive computer
algebra solution to Henrion's problem for smooth quartic curves.

\begin{Example} 
\label{ex:vinnikov}
 The following smooth quartic is a Vinnikov curve:
$$
f(x,y,z) \quad = \quad  2x^4 + y^4 + z^4-3 x^2 y^2-3 x^2 z^2+y^2 z^2.
$$
Running the algorithm in Section 2, we find that the coefficients of the
 $28$ bitangents are expressed in radicals over $\Q$. However, only
 four of the bitangents are real. 
Using Theorem \ref{thm:IDR} below, we conclude that 
there exists a real matrix representation \ref{intro:LMR}
with entries expressed in radicals over $\Q$. One such representation~is
\begin{equation}
\label{ex:hv}
f(x,y,z) \quad = \quad
{\rm det} \begin{pmatrix}
                       ux + y     &  0  &      a z &      b z \\
                          0   &     u x  - y &       cz &       d z \\
                        a z &     c  z  &   x + y  &    0  \\
                        b z  &    d  z  &     0   &   x - y 
 \end{pmatrix} \qquad \qquad \text{with} 
 \end{equation}
 $$  \begin{matrix} a &=&  -0.57464203209296160548032752478263071485849363449367...,\\
b &=&   1.03492595196395554058118944258225904539129257996969..., \\
c &=&   0.69970597091301262923557093892256027951096114611925..., \\
d &=&   0.4800486503802432010856027835498806214572648351951..., \\
u &=& \, \sqrt{2} \,\,\, =   \,\, 1.4142135623730950488016887242096980785696718....
\end{matrix} 
$$
The expression in radicals is given by the following maximal ideal in $\Q[a,b,c,d,u]$:
$$ \begin{matrix} \bigl\langle
u^2-2 \,,\,\,
256 d^8 - 384 d^6 u{+}256 d^6{-}  384 d^4 u{+}672 d^4{-}336 d^2 u{+}448 d^2{-}84 u{+}121,  \\
\,\,\,\,\,\,\,
23 c + 7584 d^7 u{+}10688 d^7{-}5872 d^5 u{-}8384 d^5{+}1806 d^3 u{+}2452 d^3{-}181 d u{-}307 d, \\
\,\,\,\,\,
23 b + 5760 d^7 u{+}8192 d^7{-}4688 d^5 u{-}6512 d^5{+}1452 d^3 u{+}2200 d^3{-}212 d u{-}232 d, \\
\,\,\,
23 a - 1440 d^7 u{-}2048 d^7{+}1632 d^5 u{+}2272 d^5{-}570 d^3 u{-}872 d^3{+}99 d u{+}81d \,
\bigr\rangle.
\end{matrix}
$$
A picture of the curve $\mathcal{V}_\R(f)$ in the affine plane $\{x=1\}$
is shown in Figure \ref{fig:vinnikov}. \qed
\end{Example}

The objective of this section is to establish the following algorithmic result:

\begin{Thm} \label{thm:IDR}
Let $f \in \Q[x,y,z]$ be a quartic
whose curve $\mathcal{V}_\C(f)$ is smooth. 
Suppose $f(x,0,0) = x^4$ and $f(x,y,0)$ is squarefree,
and let $K$ be
the splitting field for its $28$ bitangents. 
Then we can compute a determinantal representation
\begin{equation}
\label{eq:IDR}
f(x,y,z)  \,\, = \,\, {\rm det} ( x I + y D + z R) 
\end{equation}
where $I$ is the identity matrix,
$D$ is a diagonal matrix,
$R$ is a symmetric matrix, and
the entries of $D$ and $R$ are expressed in radicals over $K$.
Moreover, there exist such matrices $D$ and $R$ with real entries if
and only if $\mathcal{V}_\R(f)$ is a Vinnikov curve containing the point
$(1:0:0)$ inside the inner oval.
\end{Thm}

The hypotheses in Theorem \ref{thm:IDR} impose no 
loss of generality. Any smooth quartic will satisfy them
after a linear change of coordinates
$(x:y:z)$ in~$\P^2$.

\begin{proof}
Using the method in Section 2, we find a first representation
$f(x,y,z) = {\rm det}(xA + yB + zC) $ over the field $K$.
However, the resulting matrices $A,B,C$ might have non-real entries.
The matrix $A$ is invertible because we have assumed
${\rm det}(xA) = f(x,0,0) =  x^4$, 
which implies $ {\rm det}(A) = 1$.

The binary form $\,f(x,y,0) = {\rm det}(xA+ yB)$
is squarefree. That assumption guarantees that the $4 {\times} 4$-matrix
 $A^{-1} B$ has four distinct complex eigenvalues.
Since its entries are in $K$, 
its four eigenvalues 
lie in a radical extension field $L$ over $K$.
By choosing a suitable basis of eigenvectors, we
find a matrix $U\in\GL_4(L)$ such that $U^{-1}A^{-1} B U$ is a diagonal matrix
$D_1 = {\rm diag}(\lambda_1,\lambda_2, \lambda_3, \lambda_4)$ over
the field~$L$.

We claim that $D_2 = U^T A U $ and $D_3 = U^T B U$ are diagonal matrices.
For each column $u_i$ of $U$ we have
$\,A^{-1} B u_i = \lambda_i u_i\,$, so $B u_i = \lambda_i A u_i $.
For $1 \leq i < j \leq 4$ this implies  $\,u_j^T B u_i = \lambda_i u_j^T A u_i \,$
and, by switching indices, we get $\,u_i^T B u_j = \lambda_j u_i^T A u_j $.
Since $B$ is symmetric, the difference of the last two expressions
is zero, and we conclude $\,(\lambda_i - \lambda_j) \cdot u_i^T A u_j \,=\,0$.
By assumption, we have $\lambda_i \not= \lambda_j$ and therefore $\,u_i^T A u_j \,= 0\,$
and $\,u_i^T B u_j = 0$.
This means that $D_2$ and $D_3$ are diagonal.

Let $D_4$ be the diagonal matrix whose entries are the reciprocals of
the square roots of the entries of $D_2$. These entries are also expressed
in radicals over $K$. Then
$\,D_4 D_2 D_4 = I\,$ is the identity matrix,
$\,D_4 D_3 D_4 = D\,$ is also diagonal, and
$$\, D_4 U^T M U D_4 \,\,= \,\, x I + y D + z R $$
is the real symmetric matrix representation required in (\ref{eq:IDR}).

In order for the entries of $D$ and $R$ to be real numbers, it is necessary
(by \cite{MR2292953}) that  $\mathcal{V}_\R(f)$ be a Vinnikov curve.
We now assume that this is the case.
The existence of a real representation (\ref{eq:IDR})  is due to Vinnikov
\cite[\S 0]{MR1225986}.
A transcendental formula for the matrix entries of $D$ and $R$ in terms of theta functions is presented
in equations (4.2) and (4.3) of \cite[\S 4]{MR2292953}.
We need to show how our algebraic construction above can be used
 to compute Vinnikov's matrices $D$ and~$R$.

 Given a quartic $f \in \Q[x,y,z]$ with leading term $x^4$, the
 identity $(\ref{eq:IDR})$ translates into a system of $14$ polynomial
 equations in $14$ unknowns, namely the four entries of $D$ and the
 ten entries of $R$. For an illustration of how to solve them see Example \ref{ex:6912}.
  We claim that these equations have at most $24
 \cdot 8 \cdot 36 = 6912$ complex solutions and all solutions are
 expressed in radicals over $K$.  Indeed, there are $36$ conjugation
 orbits, and per orbit we have the freedom to transform (\ref{eq:IDR})
 by a matrix $T$ such that $T^T T = I$ and $T^T D T $ is
 diagonal.  Since the entries entries of $D$ are distinct,
 these constraints imply that $T$ is a permutation matrix
 times a diagonal matrix with entries $\pm 1$.   There
 are $24 \cdot 16$ possible choices for $T$, but $T $ and $-T$ yield
 the same triple $(I,D,R)$, so the number of solutions per
 orbit is $24 \cdot 8$.

 We conclude that, for each of the $36$ orbits, either all
 representations (\ref{eq:IDR}) are real or none of them is.  Hence,
 by applying this method to all $36$ inequivalent symmetric linear
 determinantal representations constructed in Section 3, we are
 guaranteed to find Vinnikov's real matrices $D$ and $R$. See also
 \cite[Section 2]{us} for additional examples and a more detailed
 discussion.
\end{proof}

The above argument for the simultaneous diagonalizability of
 $A$ and $B$ is taken from Greub's linear algebra text book \cite{MR0369382}.
 We could also handle the exceptional case when
$A^{-1} B$ does not have four distinct eigenvalues.
Even in that case  there exists a matrix $U$ in radicals over $K$ such that
$U^T A U$ and $U^T B U$ are diagonal, but the construction of $U$ is
more difficult. The details are found  in \cite[\S IX.3]{MR0369382}.

\begin{Cor}
Every smooth Vinnikov curve
has a real determinantal representation (\ref{intro:LMR})
in radicals over the splitting field $K$ of its $28$ bitangents.
\end{Cor}

We close with the remark that the representation
(\ref{eq:IDR}) generally does not exist over the field $K$ itself
but the passage to a radical extension field is necessary.

\begin{Example} \label{ex:6912}
All $6912$ matrix representations $xI  + yD + zR$
of the Edge quartic $\, E(x,y,z)  = 
25 \cdot (x^4+y^4+z^4)\,-\, 34 \cdot (x^2 y^2+x^2 z^2+y^2 z^2)$
are non-real and have degree~$4$ over $\Q$.
 The entries of $D$ are the four complex zeros
of the irreducible polynomial $\, x^4 - \frac{34}{25} x^2 + 1$.
After fixing $D$, we have $192$ choices for $R$, namely,
selecting one of the $36$ orbits fixes $R$ up to conjugation by
${\rm diag}(\pm 1, \pm 1, \pm 1, \pm 1 )$. 
For the orbit of the matrix
$xA {+} yB {+} zC$ in (\ref{LMRforedge}), our algorithm gives the
representation
\[
\begin{small}
  D \,\, = \,\,
  \begin{pmatrix}
    -\sqrt{21}/5 - 2i/5 & 0 & 0 & 0\\
    0 & \sqrt{21}/5+  2i/5& 0 & 0\\
    0 & 0 &-\sqrt{21}/5 + 2i/5 & 0\\
    0 & 0 & 0 & \sqrt{21}/5 - 2i/5
  \end{pmatrix}
\end{small}
\]
\[
\begin{small}
R \,\,=\,\,
  \begin{pmatrix}
    0 &-\frac{2}{5} (\sqrt{3/7}+i)& -\sqrt{27/35}& 0\\
\! \! -\frac{2}{5} (\sqrt{3/7}+i)& 0 &0 & \sqrt{27/35}\\
   -\sqrt{27/35} & 0 & 0 & \! -\frac{2}{5}(\sqrt{3/7}-i)\! \\
    0& \sqrt{27/35} & -\frac{2}{5} (\sqrt{3/7}-i)& 0
  \end{pmatrix}.
\end{small}
\]
\qed
\end{Example}

\section{Sums of Three Squares and Steiner Complexes}\label{sec:SOS}

Our next goal is to write the given quartic $f$
as the sum of three squares of quadrics.
Such representations (\ref{intro:SOS}) are classified by Gram matrices of rank $ 3$.
A \emph{Gram matrix} for $f$ is a symmetric $6\times 6$ matrix $G$ 
with entries in $\C$ such that 
$$ f  \,\,\,= \,\,\, v^T \cdot G \cdot v \quad \hbox{where} \quad v \,= \, (x^2,y^2,z^2,xy,xz,yz)^T . $$
We can write $\,G = H^T \cdot H$, where $H$ is an $r \times 6$-matrix
and $r = {\rm rank}(G)$.  Then the factorization $\,f\,=\,(Hv)^T \cdot
(Hv)\,$ expresses $f$ as the sum of $r$ squares.

It can be shown that no Gram matrix with $r \leq 2$ exists when $f$ is smooth, and there are
infinitely many for $r \geq 4$.  For $r=3$ their number is $63$ by Theorem
\ref{thm:sixtythree}.

Gram matrices classify the representations (\ref{intro:SOS}):
two distinct representations
$$ f \,\, = \,\, q_1^2+q_2^2+q_3^2 \,\, = \,\, p_1^2+p_2^2+p_3^2 $$ 
correspond to the same Gram
matrix $G$ of rank $3$ if and only if there exists an orthogonal matrix
$T\in{\rm O}_3(\C)$ such that $\, T \cdot (p_1,p_2,p_3)^T=(q_1,q_2,q_3)^T$.
The objective of this section is to present an algorithmic proof 
for the following result.

\begin{Thm}
\label{thm:sixtythree}
Let $f \in \Q[x,y,z]$ be a smooth quartic and
$K$ the splitting field for its $\,28$ bitangents.
Then $f$ has precisely $\,63$ Gram matrices of rank $\,3$,
all of which we compute using rational arithmetic over the field $K$.
\end{Thm}

The fact that $f$ has $63$ Gram matrices of rank $3$ is a known result due to
Coble \cite[Ch.~1, \S14]{Cob}; see also \cite[Prop.~2.1]{PRSS}. Our contribution is 
a new proof that yields a $K$-rational algorithm for computing all rank-3
Gram matrices. Instead of appealing
to the Jacobian threefold of $f$, as in \cite{PRSS}, we shall identify
the $63$ Gram matrices with the $63$
 Steiner complexes of bitangents (see \cite[\S VI]{MR0115124} and
 \cite[\S 6]{Dol}).

 We begin by constructing a representation
 $f=q_1^2+q_2^2+q_3^2$ from any pair of bitangents. Let $\ell,\ell'$
 be distinct bitangents of $f$, and let $p \in \C[x,y,z]_2$ be a
 non-singular quadric passing through the four contact points of
 $\ell\ell'$ with $f$.  By Max Noether's Fundamental Theorem  \cite[\S~5.5]{MR1042981},
 the ideal  $\big\langle \ell \ell', f \bigr\rangle$ contains $p^2$, thus
\begin{equation}\label{eq:firstQMR}
f \,\,= \,\, \ell\ell'u-p^2,
\end{equation}
for some quadric 
$u \in \C[x,y,z]_2$, after rescaling  $p$ by a constant.  Over $\C$, the identity
\eqref{eq:firstQMR} translates directly into one of the form:
\begin{equation}\label{eq:secondQMR}
f \,\, = \,\, \left(\frac{1}{2} \ell\ell'+\frac{1}{2} u\right)^2+\left(\frac{1}{2i}
  \ell\ell' -\frac{1}{2i} u\right)^2+ (i p)^2.
\end{equation}

\begin{Remark}
Just as systems of contact cubics to $\sV_\C(f)$
were behind the formula (\ref{intro:LMR}),
systems of contact conics to $\sV_\C(f)$ are responsible for
the representations (\ref{intro:SOS}). The simplest choice of a contact
conic is a product of two bitangents. 
\end{Remark}

In (\ref{eq:secondQMR}) we wrote $f$ as a sum of three squares
over $\C$. There are $\binom{28}{2}=378$ pairs
$\{\ell,\ell'\}$ of bitangents. We will see 
Theorem \ref{thm:Steiner}
that each pair forms a syzygetic quadruple 
with $5$ other pairs. This yields $378/6=63$ equivalence classes. More importantly,
there is a combinatorial rule for determining these $63$ classes from a 
Cayley octad. This allows us to compute
the $63$ Gram matrices over~$K$.

\smallskip

Equation \eqref{eq:firstQMR} can also be read as a
quadratic determinantal representation 
\begin{equation}
\label{eq:thirdQMR}
f\;=\;\det\left(\begin{matrix} q_0&q_1 \\
    q_1&q_2 \end{matrix} \right)
\end{equation}
with $q_0=\ell\ell'$, $q_1=p$, and $q_2=u$.  This expression 
gives rise to the quadratic system of contact conics 
$\{\lambda_0^2 q_0 +2 \lambda_0\lambda_1 q_1+\lambda_1^2 q_2^2\;:\; \lambda \in \P^1(\C)\}$.  
The implicitization of this quadratic system is a quadratic form on ${\rm span}\{q_0, q_1, q_2\}$. 
With respect to the basis $(q_0,q_1,q_2)$, it is represented by a symmetric $3\times3$ matrix $C$.
Namely, 
$$  \quad \qquad C \,=\,\begin{pmatrix} 0&0&2\\0&-1&0\\2&0&0 \end{pmatrix}
\quad \hbox{and its inverse is} \quad
C^{-1} \, = \, \begin{pmatrix} 0 & 0 &1/2\\ 0 & -1&0\\1/2&0&0\end{pmatrix} \! . $$
The formula (\ref{eq:thirdQMR}) shows that     
$\,f \,=\, q_0 q_2 - q_1^2 \,=\, (q_0,q_1,q_2) \cdot C^{-1} \cdot (q_0,q_1,q_2)^T $.
We now  extend $q_0,q_1,q_2$ to a basis $q = (q_0,q_1,q_2,q_3,q_4, q_5)$ of
$\C[x,y,z]_2$. Let $T$ denote the matrix that takes the monomial basis
$v = (x^2,y^2,z^2,xy,xz,yz)$ to $q$. If $\wt G$ is the $6 \times 6$ matrix with $C^{-1}$ 
in the top left block and zeros elsewhere, then
\begin{equation}
\label{eq:fourthQMR}
f \quad = \quad 
(q_0,q_1,q_2) \cdot C^{-1} \cdot (q_0,q_1,q_2)^T  \quad = \quad
   v^T \cdot T^T \cdot \wt G \cdot T \cdot v.
\end{equation}
Thus, $\,G=T^T\wt G T\,$ is a rank-3 Gram
matrix of $f$. This construction is completely reversible, showing
that every rank-3 Gram matrix of $f$ is obtained in this way.

The key player in the formula (\ref{eq:fourthQMR}) is the quadratic form given by $C$.
From this, one easily gets the Gram matrix $G$.
We shall explain how to find $G$ geometrically from the pair of bitangents $\ell,\ell'$.
The following result is taken from
Salmon \cite{MR0115124}:

\begin{Prop}\label{prop:ContactConics}
  Let $f=\det(Q)$ where $Q$ is a symmetric $2\times 2$-matrix with
  entries in $\C[x,y,z]_2$ as in (\ref{eq:thirdQMR}).  Then $Q$ defines a quadratic
  system of contact conics $\,\lambda^TQ\lambda\,$,
  $\lambda\in \P^1(\C)$, that contains exactly six products of two bitangents.
\end{Prop}

\begin{proof}[Sketch of Proof]
  To see that $\lambda^TQ\lambda\,$
  is a contact conic, note that for any
  $\lambda,\mu \in \C^2$,
 \begin{equation}
 \label{eq:Mmess}
 (\lambda^TQ\lambda)(\mu^TQ\mu) - (\lambda^TQ\mu)^2 \,=\,
  \sum_{i,j,k,l} \lambda_i\lambda_j \mu_k \mu_l ( Q_{ij}Q_{kl} -  Q_{ik}Q_{jl}).
  \end{equation}
  The expression $\,Q_{ij}Q_{kl}-Q_{ik}Q_{jl}\,$ is a multiple of
  $\det(Q)=f$, and hence so is the left hand side of (\ref{eq:Mmess}).
  This shows that $ \lambda^TQ\lambda$ is a contact conic of
  $\sV_\C(f)$.  The set of singular conics is a cubic hypersurface in
  $\C[x,y,z]_2$. As $ \lambda^TQ\lambda$ is quadratic in $\lambda$, we
  see that there are six points $\lambda \in \P^1(\C)$ for which $
  \lambda^TQ\lambda$ is the product of two linear forms. These are
  bitangents of $f$ and therefore $K$-rational.
  \end{proof}
  
\begin{Remark}\label{rem:PSDConic2PSDGram}
 If the Gram matrix $G$ is real, then it is positive (or negative) semidefinite if and only if the quadratic system $\sQ=\{\lambda^T Q\lambda\;|\;\lambda\in\P^1(\C)\}$ does not contain any real conics. For if $G$ is real, we may take a real basis $(q_0',q_1',q_2')$ of ${\rm span}\{q_0,q_1,q_2\}=\ker(G)^\perp$ in $\C[x,y,z]_2$. If $\sQ$ does not contain any real conics, then the matrix $C'$ representing $\sQ$ with respect to the basis $(q_0',q_1',q_2')$ is definite. Using $C'$ instead of $C$ in the above construction, we conclude that $C^{\prime -1}$ is definite and hence $G$ is semidefinite. The converse follows by reversing the argument.
\end{Remark}

We now come to Steiner complexes,
the second topic in the section title.
 
\begin{Thm} \label{thm:Steiner}
Let $\, \mathcal{S} = \bigl\{\{\ell_1,\ell_1'\}, \hdots,\{\ell_6,\ell_6'\}\bigr\} \,$ be six pairs of
  bitangents of a smooth quartic $f \in \Q[x,y,z]$. Then the following 
  three conditions are equivalent:
\begin{enumerate} \item 
The reducible quadrics
  $\ell_1\ell_1', \hdots, \ell_6\ell_6'$ lie in a system of contact conics 
  $\,\lambda^TQ\lambda ,$ $\,\lambda\in\P^1(\C)$, for $Q$ a quadratic
  determinantal representation (\ref{eq:thirdQMR}) of $f$.
\item For each $i\neq j$, the eight contact points
  $\sV_\C(\ell_i\ell_i'\ell_j\ell_j')\cap\sV_\C(f)$ lie on a conic.
\item With indices as in the bitangent matrix 
(\ref{BitangentMatrix}) for a Cayley octad, either
$$ \begin{matrix}
 \mathcal{S} =  \,\bigl\{\{b_{ik},b_{jk}\}\,|\, \{i,j\}=I\text{ \ and\  }k\in I^c\,\bigr\} &
\text{for a $2$-set $I \subset \{1,\ldots, 8\}$,} \\
 \,\,\,\, or \,\,\mathcal{S}  = 
\bigl\{\{b_{ij},b_{kl}\}\,|\, \{i,j,k,l\}=I\text{ or }\{i,j,k,l\}=I^c\bigr\} \! &
\text{for a $4$-set $I \subset \{1,\hdots, 8\}$}.
\end{matrix}
$$
\end{enumerate}
\end{Thm}

\begin{proof}
This is a classical result due to Otto Hesse \cite{049.1317cj}. 
The proof can
also be found in the books of Salmon \cite{MR0115124} and Miller-Blichfeldt-Dickson
\cite[\S 185--186]{MR0123600}.
\end{proof}

A \emph{Steiner complex} is a sextuple $\mathcal{S}$ of pairs of
bitangents satisfying the conditions of Theorem~\ref{thm:Steiner}. A
pair of bitangents in $\mathcal{S}$ is either of the form $\{b_{ik},b_{jk}\}$ (referred
to as type $\bigvee$) or of the form $\{b_{ij},b_{kl}\}$ (type
$||$). The first type of Steiner complex in Theorem \ref{thm:Steiner} (3) 
contains pairs of bitangents of type $\bigvee$ and the
second type contains pairs of type $||$. There are
$\binom{8}{2}=28$ Steiner complexes  of type $\bigvee$
and $\binom{8}{4}/2=35$ Steiner complexes of bitangents of type
$||$. The two types of Steiner complexes are easy to remember by the
following combinatorial pictures:

\medskip
\begin{center}
  \begin{minipage}{0.4\linewidth}
    \begin{center}
      \begin{tikzpicture}
        \fill (1.5,1) circle (2pt); \fill (3.5,1) circle (2pt); \fill
        (0,0) circle (2pt); \fill (1,0) circle (2pt); \fill (2,0)
        circle (2pt); \fill (3,0) circle (2pt); \fill (4,0) circle
        (2pt); \fill (5,0) circle (2pt); \draw (1.5,1) -- (0,0); \draw
        (3.5,1) -- (0,0); \draw (1.5,1) -- (1,0); \draw (3.5,1) --
        (1,0); \draw[line width=1.5pt] (1.5,1) -- (2,0); \draw[line width=1.5pt]
        (3.5,1) -- (2,0); \draw (1.5,1) -- (3,0); \draw (3.5,1) --
        (3,0); \draw (1.5,1) -- (4,0); \draw (3.5,1) -- (4,0); \draw
        (1.5,1) -- (5,0); \draw (3.5,1) -- (5,0);
      \end{tikzpicture}\\
      Type $\bigvee$\hspace{.5em}
    \end{center}
  \end{minipage}
  \begin{minipage}{0.4\linewidth}
    \begin{center}
      \begin{tikzpicture}
        \fill (0,1) circle (2pt); \fill (1,1) circle (2pt); \fill
        (2,1) circle (2pt); \fill (3,1) circle (2pt); \fill (0,0)
        circle (2pt); \fill (1,0) circle (2pt); \fill (2,0) circle
        (2pt); \fill (3,0) circle (2pt); \draw[line width=1.5pt] (0,1) --
        (1,1); \draw (0,1) -- (1,0); \draw (0,1) -- (0,0); \draw (0,0)
        -- (1,1); \draw[line width=1.5pt] (0,0) -- (1,0); \draw (1,0) --
        (1,1); \draw (2,1) -- (3,1); \draw (2,1) -- (3,0); \draw (2,1)
        -- (2,0); \draw (2,0) -- (3,1); \draw (2,0) -- (3,0); \draw
        (3,0) -- (3,1);
      \end{tikzpicture}\\
      Type $||$\hspace{.5em}
    \end{center}
  \end{minipage}
\end{center}

\smallskip

This combinatorial encoding of Steiner complexes
enables us to derive the last column in 
Table \ref{table:sixtypes} in the Introduction.
We represent the quartic as (\ref{intro:SOS}) with
$A,B,C$ real, as in \cite{MR1225986}.
The corresponding Cayley octad 
$\{O_1,\ldots, O_8\}$ is invariant under complex conjugation.
Let $\pi$ be the permutation in $S_8$ that represents complex conjugation, 
 meaning $\overline{O_i}=O_{\pi(i)}$.  Then complex conjugation on the
$63$ Steiner complexes is given by the action
of $\pi$ on their labels. For instance, when all $O_i$ are real, as in the first row of
Table~\ref{table:sixtypes}, then $\pi$ is the identity. For the other rows we can relabel so that
$\, \pi = (1 2),\, \pi = (1 2)(34),\, \pi = (12)(34)(56)\,$ and $\, \pi =(12)(34)(56)(78)$.
We say that
a Steiner complex $\mathcal{S}$ is {\em real}
if its labels are fixed under $\pi$.
For example, if $\mathcal{S}$ is the Steiner complex
$\{\!\{b_{13},b_{23}\},$\dots$,\{b_{18},b_{28}\}\!\}$ of type $\bigvee$ as
above, then $\mathcal{S}$ is real if and only if
$\pi$ fixes $\{1,2\}$.  Similarly, if $\mathcal{S}$
is the Steiner complex $\{\{b_{12},b_{34}\}$, $\{b_{13},b_{24}\}$,
$\{b_{14},b_{23}\}$, $\{b_{56},b_{78}\}$, $\{b_{57},b_{68}\}$,
$\{b_{58},b_{67}\}\}$ of type $||$, then $\mathcal{S}$ is real if and only if
$\pi$ fixes the partition $\bigl\{\!\{1,2,3,4\}, \{5,6,7,8\}\!\bigr\}$.
For instance, for the empty curve, in the last row Table~\ref{table:sixtypes}, one can check that exactly
$15$ Steiner complexes are fixed by $ \pi =(12)(34)(56)(78)$, as listed in Section~\ref{sec:Gram}.

\smallskip

We now sum up what we have achieved in this section,
namely, a recipe for constructing the $63$ Gram matrices
from the $28+35$ Steiner complexes
$\bigvee$ and $||$.

\begin{proof}[Proof and Algorithm for Theorem \ref{thm:sixtythree}]
  We take as input a smooth ternary quartic $f\in\Q[x,y,z]$ and any of
  the $63$ Steiner complexes $\bigl\{ \! \{\ell_1,\ell'_1\},\dots,
  \{\ell_6,\ell'_6\}\!\bigr\}$ of bitangents of $\sV_\C(f)$. From this
  we can compute a rank-$3$ Gram matrix for $f$ as follows. The six
  contact conics $\ell_i \ell'_i$ span a $3$-dimensional subspace of
  $K[x,y,z]_2$, by Theorem~\ref{thm:Steiner} (1), of which
  $\{\ell_1\ell_1',\ell_2\ell_2', \ell_3\ell_3'\}$ is a basis.  The
  six vectors $\ell_i \ell_i'$ lie on a conic in that subspace, and we
  compute the symmetric $3 {\times} 3$-matrix $\wt C$ representing
  this conic in the chosen basis.  We then extend its inverse $\wt
  C^{-1}$ by zeroes to a $6 {\times} 6$ matrix $\wt G$ and fix an
  arbitrary basis $\{q_4,q_5,q_6\}$ of ${\rm
    span}\{\ell_1\ell_1',\ell_2\ell_2', \ell_3\ell_3'\}^{\perp}$ in
  $K[x,y,z]_2$. Let $T \in K^{6 \times 6}$ be the matrix taking $v =
  (x^2,y^2,z^2,xy,xz,yz)^T$ to $(\ell_1\ell_1',\ell_2\ell_2',
  \ell_3\ell_3',q_4,q_5,q_6)^T$. Then $\,G= T^T\wt G T\,$ is the
  desired rank-$3$ Gram matrix for $f$, and all rank-$3$ Gram matrices
  arise in this way.  Note that $G$ does not depend on the choice of
  $q_4, q_5, q_5$.
\end{proof}
    
\begin{Remark}
  Given $f$, finding a Steiner complex as input for the above
  algorithm is not a trivial task. But when a linear determinantal
  representation of $f$ is known, and thus a Cayley octad, one can use
  the criterion in Theorem~\ref{thm:Steiner} (3).
\end{Remark}

\begin{Example} \label{ex:BerndQuartic}
We consider the quartic $ f = {\rm det}(M)$ defined by the matrix
$$ 
M \,=\, \begin{pmatrix}
    52 x + 12 y - 60 z  & -26 x - 6 y + 30 z &  48 z  & 48 y \\
 \!  -26 x - 6 y + 30 z & 26 x + 6 y - 30 z & \! -6 x {+} 6 y {-} 30 z &    -45 x {-} 27 y {-} 21 z \! \\
     48 z & -6 x + 6 y - 30 z & -96 x & 48 x \\
    48 y & -45 x - 27 y - 21 z & 48 x & -48 x 
    \end{pmatrix} \! .
        $$
 The complex curve $\mathcal{V}_\C(f)$ is smooth and its set of real points
 $\mathcal{V}_\R(f)$ is empty.
The corresponding Cayley octad consists of four pairs of complex conjugates:
$$
O^T \,\, = \,\, \begin{pmatrix}
   i &  - i &  0 &  0 &  -6+4  i &  -6-4  i &  3+2  i &  3-2  i  \\
 \! 1+ i &  1- i &  0 &  0 &  -4+4  i &  -4-4  i &  7- i &  7+ i  \\
  0 &  0 &   i &  - i &  -3+2  i &  -3-2  i &  -\frac{86}{39}- \frac{4}{13}  i &  -\frac{86}{39}+\frac{4}{13}  i  \\
  0 &  0 &  1+ i &  1- i &  1- i &  1+ i &  \frac{4}{39}-\frac{20}{39}  i &  \frac{4}{39}+\frac{20}{39}  i  \\
\end{pmatrix} \! .
$$
Here the $8 {\times} 8$ bitangent matrix
$\,O M O^T = (b_{ij}) \,$ is defined over the field
  $K = \Q(i)$ of Gaussian rationals, and hence so are all 
  $63$ Gram matrices.
  According to the lower right entry in Table \ref{table:sixtypes},
  precisely $15$ of the Gram matrices are real, and hence these
  $15$ Gram matrices have their entries in $\Q$. For instance, the representation
$$
f \,\, = \,\, 288
 \begin{pmatrix}
x^2 \\ y^2 \\ z^2 \\ x y \\ x z \\ y z \end{pmatrix}^{\!\!\! T } \!
\begin{pmatrix}
	45500 & 3102 & -9861 & 5718 & -9246 & 4956 \\
	3102 & 288 & -747 & 882 & -18 & -144 \\
	-9861 & -747 & 3528 & -864 & -1170 & -504 \\
	5718 & 882 & -864 & 4440 & 1104 & -2412 \\
	-9246 & -18 & -1170 & 1104 & 11814 & -5058 \\
	4956 & -144 & -504 & -2412 & -5058 & 3582
               \end{pmatrix} \!
  \begin{pmatrix}
x^2 \\ y^2 \\ z^2 \\ x y \\ x z \\ y z \end{pmatrix}
$$
is obtained by applying our algorithm for Theorem \ref{thm:sixtythree}
to the Steiner complex
  $$ \mathcal{S} \,= \,
  \bigl\{\!\{ b_{13}, b_{58} \},\,
             \{ b_{15}, b_{38} \}, \, 
             \{ b_{18}, b_{35} \}, \,
             \{ b_{24}, b_{67} \},\,
            \{ b_{26}, b_{47} \},\,
            \{ b_{27}, b_{46} \}\! \bigr\} .$$
The above Gram matrix has rank $3$ and is positive semidefinite, so
it translates into a representation
 (\ref{intro:SOS}) for $f$ as the sum of three squares of quadrics over $\R$.
\qed
\end{Example}

\section{The Gram spectrahedron}\label{sec:Gram}

The {\em Gram spectrahedron} ${\rm Gram}(f)$ of a real ternary quartic $f$  is the set
of its positive semidefinite Gram matrices. This spectrahedron is
the intersection of the cone of positive semidefinite $6 {\times} 6$-matrices
with a $6$-dimensional affine subspace.
By Hilbert's result in
\cite{MR1510517}, ${\rm Gram}(f)$ is non-empty if and only if $f$ is non-negative.
In terms of coordinates on the $6$-dimensional subspace given by a fixed quartic 
\[
f(x,y,z) \,\,\, = \,\,\,  c_{400} x^4 + c_{310} x^3 y + c_{301} x^3 z + 
c_{220} x^2 y^2 + c_{211} x^2 y z + \cdots + c_{004} z^4 , 
\]
the Gram spectrahedron ${\rm Gram}(f)$ is the set of all positive semidefinite matrices
\begin{equation}\label{eq:GramMatrix}
\begin{small}
 \begin{pmatrix}
   c_{400} & \lambda_1 & \lambda_2 & \frac 12 c_{310} & \frac 12
   c_{301} & \lambda_4\\[.2em]
   \lambda_1 & c_{040} & \lambda_3 & \frac 12 c_{130} & \lambda_5 &
   \frac 12 c_{031}\\[.2em]
   \lambda_2 & \lambda_3 & c_{004} & \lambda_6 & \frac 12 c_{103} &
   \frac 12 c_{013}\\[.2em]
   \frac 12 c_{310} & \frac 12 c_{130} & \lambda_6 & c_{220} -
   2\lambda_1 & \frac 12 c_{211}-\lambda_4 & \frac 12
   c_{121}-\lambda_5\\[.2em]
   \frac 12 c_{301} & \lambda_5 & \frac 12 c_{103} & \frac 12
   c_{211}-\lambda_4 & c_{202}-2\lambda_2 & \frac 12 c_{112}-\lambda_6\\[.2em]
   \lambda_4 & \frac 12 c_{031} & \frac 12 c_{013} & \frac 12
   c_{121}-\lambda_5 & \frac 12 c_{112} - \lambda_6 & c_{022} -
   2\lambda_3
 \end{pmatrix}
  \end{small} \! ,
\,\, \text{where $\lambda \in \R^6$.}
\end{equation}
The main result of \cite{PRSS} is that a smooth positive quartic $f$
has exactly eight inequivalent representations as a sum of three real
squares, which had been conjectured in \cite{PR}. These eight
representations correspond to rank-$3$ positive semidefinite Gram
matrices.  We call these the \emph{vertices of rank $3$} of ${\rm
  Gram}(f)$.  In Section 5 we compute them using arithmetic over $K$.
 
 We define the {\em Steiner graph} of the Gram spectrahedron 
 to be the graph on the eight vertices of rank $3$ whose edges represent edges of
 the convex body ${\rm Gram}(f)$.
  
\begin{Thm}\label{thm:TwoK4s}
  The Steiner graph of the Gram spectrahedron ${\rm Gram}(f)$ of a generic positive ternary quartic $f$
  is the disjoint union $K_4 \sqcup K_4$ of two complete graphs, and the relative interiors of these edges consist of rank-$5$ matrices.
  \end{Thm}

 This theorem means that the eight rank-$3$ Gram matrices are divided into two groups of four,
and, for $G$ and $G'$ in the same group, we have
$\rank(G+G') \leq 5$. The second sentence asserts that 
$\rank(G+G') = 5$ holds for generic $f$. For the proof it suffices
to verify this for one specific $f$.
This we have done, using exact arithmetic,
for the quartic in Example \ref{ex:BerndQuartic}. For instance, the
rank-$3$ vertices
\[
\begin{array}{cc}
(\frac{1}{288})G=& (\frac{1}{288})G'=\\
\hspace{-.2cm}\begin{tiny}
\begin{pmatrix}
45500 & 3102 & -9861 & 5718 & -9246 & 4956 \\
3102 & 288 & -747 & 882 & -18 & -144 \\
-9861 & -747 & 3528 & -864 & -1170 & -504 \\
5718 & 882 & -864 & 4440 & 1104 & -2412 \\
-9246 & -18 & -1170 & 1104 & 11814 & -5058 \\
4956 & -144 & -504 & -2412 & -5058 & 3582
\end{pmatrix}\end{tiny}& \hspace{-.4cm}\begin{tiny}
\begin{pmatrix}45500 & -2802 & -6666 & 5718 & -9246 & 132 \\
-2802 & 288 & -72 & 882 & 1206 & -144 \\
-6666 & -72 & 3528 & -4878 & -1170 & -504 \\
5718 & 882 & -4878 & 16248 & 5928 & -3636 \\
-9246 & 1206 & -1170 & 5928 & 5424 & -1044 \\
132 & -144 & -504 & -3636 & -1044 & 2232
\end{pmatrix}\end{tiny}
\end{array}\]
both contain the vector $(11355,-4241, 47584, 8325, 28530,  36706)^T$ in their
kernel, so that $\rank(G+G') \leq 5$. But this vector spans the intersection
of the kernels, hence $\rank(G+G') = 5$, and every matrix
on the edge has rank~$5$.

We also know that
there exist instances of smooth positive quartics where the rank along an edge drops to $4$.
One such example is the Fermat quartic, $x^4+y^4+z^4$, which has two psd rank-$3$ Gram
matrices whose sum has rank~$4$. 
We do not know whether the Gram spectrahedron ${\rm Gram}(f)$ has  proper faces 
of dimension $\geq 1$ other than the twelve edges in  the
Steiner graph $K_4 \sqcup K_4$.
In particular, we do not know whether the
Steiner graph coincides with the graph of all edges of ${\rm Gram}(f)$.

\begin{proof}[Proof of Theorem~\ref{thm:TwoK4s}]
Fix a real symmetric linear determinantal representation
$M=xA+yB+zC$ of $f$. The existence of such $M$ when $f$
is positive was proved by Vinnikov \cite[\S0]{MR1225986}. 
The Cayley octad $\{O_1,\ldots,O_8\}$ determined by $M$
 consists of four pairs of complex conjugate
points. Recall
from Section \ref{sec:SOS} that a Steiner complex corresponds to
either a subset $I \subset \{1,\ldots,8\}$ with $|I|=2$ (type
$\bigvee$) or a partition $\, I | I^c\,$
 of $\{1,\ldots,8\}$ into two subsets of size $4$
  (type $||$). We write $\sS_I$ for the Steiner
complex given by $I$ or $I|I^c$ and $G_I$ for the corresponding
Gram matrix. Theorem~\ref{thm:TwoK4s} follows from the more
precise result in Theorem \ref{Thm:TwoK4sIndices} which we shall prove further below.
 \end{proof}

\begin{Thm}\label{Thm:TwoK4sIndices} 
  Let $f$ be positive with $\sV_{\C}(f)$ smooth and conjugation acting on the Cayley octad by 
  $\,\ol O_i=O_{\pi(i)}\,$ for
  $\,\pi=(12)(34)(56)(78)$. The eight Steiner complexes corresponding to the
  vertices of rank $3$ of the Gram spectrahedron ${\rm Gram}(f)$ are
   $$ \begin{matrix}
  1357|2468 & 1368|2457 & 1458|2367 & 1467|2358\\
  1358|2467 & 1367|2458 & 1457|2368 & 1468|2357
  \end{matrix}
$$
The Steiner graph $K_4 \sqcup K_4$ is given by pairs of Steiner complexes in the same row.
\end{Thm}

Our proof of Theorem~\ref{Thm:TwoK4sIndices} consists of two parts:
(1) showing that the above Steiner complexes 
give the positive semidefinite Gram matrices and (2) showing how they form two copies of $K_4$. We
will begin by assuming (1) and proving (2):

By Theorem~\ref{thm:Steiner}, for any two pairs of bitangents 
$\{\ell_1,\ell_1'\}$ and $\{\ell_2,\ell_2'\}$ in a fixed Steiner complex $\sS$, there is a 
conic $u$ in $\P^2$ that passes through the eight contact points of these four
bitangents with $\sV_\C(f)$.  In this manner,  one 
associates with every Steiner complex $\sS$ a set of $\binom{6}{2}=15$
conics, denoted $\conics(\sS)$.

\begin{Lemma} \label{lem:conicsmeet}
  Let $\sS$ and $\sT$ be Steiner complexes with Gram matrices $G_{\sS}$
  and $G_{\sT}$.  If $\conics(\sS)\cap\conics(\sT)\neq \emptyset$ then $\rank(G_{\sS}+G_{\sT})\leq
  5$.
\end{Lemma}

\begin{proof} 
Suppose $\sS = \{\{\ell_1,\ell_1'\},\hdots, \{\ell_6,\ell_6'\}\}$.
Let $Q$ be a quadratic matrix
representation (\ref{eq:thirdQMR}) such that the six points $\ell_1\ell_1', \hdots,
\ell_6\ell_6' \in \P(\C[x,y,z]_2)$
lie on the conic $\{\lambda^TQ\lambda\;:\;\lambda \in \P^1(\C)\}$.
By the construction in the proof of Theorem~\ref{thm:sixtythree},
we know that the projective plane in 
$\P(\C[x,y,z]_2)$ spanned by this conic is $\ker(G_{\sS})^{\perp}$.

Consider two pairs $\{\ell_1,\ell_1'\},\{\ell_2,\ell_2'\}$ from $\sS$
and let $u \in \conics(\sS)$ be the unique conic passing through the
eight contact points of these bitangents with the curve $\sV_{\C}(f)$. By our choice of
$Q$, we can find $\lambda, \mu \in \P^1$ such that
$\lambda^TQ\lamda=\ell_1\ell_1'$ and $\mu^TQ\mu = \ell_2\ell_2'$. 
Equation~(\ref{eq:Mmess}) then shows that $u=\lambda^TQ\mu$. From this we see 
that $u\in {\rm span}\{Q_{11},Q_{12},Q_{22}\} =\ker(G_{\sS})^{\perp}$. 
Therefore, $\conics(\sS) \subseteq \ker(G_{\sS})^{\perp}$.

If $\conics(\sS) \cap \conics(\sT) \neq \emptyset$, then the two 3-planes
$\ker(G_{\sS})^{\perp}$ and $\ker(G_{\sT})^{\perp}$ meet nontrivially. Since
$\C[x,y,z]_2$ has dimension $6$, this implies that $\ker(G_{\sS})$ and
$\ker(G_{\sT})$ meet nontrivially. Hence $\rank(G_{\sS}+G_{\sT})\leq 5$. 
\end{proof}

For example, 
$\conics(\sS_{1358})$ and $ \conics(\sS_{1457})$ share the conic going through 
the contact points of $b_{15},b_{26},b_{38},$ and $b_{47}$. 
Lemma~\ref{lem:conicsmeet} then implies
 $\rank(G_{1358}+G_{1457}) \leq 5$, as shown above for Example~\ref{ex:BerndQuartic} with $G = G_{1358}$ and $G'=G_{1457}$.
 
 Using this approach, we only have to check that
$\conics(\sS_I)\cap\conics(\sS_J)\neq\emptyset$ when $I$ and $J$ are in the
same row of the table in Theorem~\ref{Thm:TwoK4sIndices}. More precisely:

\begin{Lemma}\label{lem: Gram+Gram rank} Let $I$ and $J$ be subsets of
  $\{1,\hdots, 8\}$ of size four with $I \neq J$ and $I\neq J^c$.  
  Then $\,\conics(\sS_{I})\cap \conics(\sS_{J}) \neq \emptyset \,$ if
  and only if $\, |I\cap J| = 2$.
\end{Lemma}

\begin{proof} 
Every syzygetic set of four bitangents $\ell_1,\ell_2,\ell_3,\ell_4$
determines a unique conic $u$ passing through their eight contact
points with $\sV_{\C}(f)$.  There are three ways to collect the four
bitangents into two pairs, so $u$ appears in $\conics(\sS)$ for exactly three
Steiner complexes. Thus for two Steiner complexes $\sS_I$ and
$\sS_J$, we have $\conics(\sS_I)\cap\conics(\sS_J)\neq \emptyset$ if and only if
there are bitangents $\ell_1,\ell_2,\ell_3,\ell_4$ such that
$\{\ell_1{,}\ell_2\},\{\ell_3{,}\ell_4\}\in \sS_I$ and
$\{\ell_1{,}\ell_3\},\{\ell_2{,}\ell_4\}\in \sS_J$. This translates 
into  $| I \cap J| = 2 $.
\end{proof}

To complete the proof of Theorem~\ref{Thm:TwoK4sIndices}, it remains
to show that the eight listed Steiner complexes give positive semidefinite
Gram matrices.
Recall that a Steiner complex $\sS_I$ is real if and only if $I$ is fixed
by the permutation $\pi$ coming from conjugation. As stated in Section
\ref{sec:octad}, there are 15 real Steiner complexes, namely,
\begin{enumerate}\label{test}
\item The eight complexes of type $||$ listed in
  Theorem~\ref{Thm:TwoK4sIndices}.
\item Three more complexes of type $||$, namely $\,1234|5678$,
  $1256|3478$, $1278|3456$.
\item Four complexes of type $\bigvee$, namely $\,12$, $34$, $56$, $78$.
\end{enumerate}
Since we know from \cite{PRSS} that exactly eight of these give positive
semidefinite Gram matrices,
it suffices to rule out the seven Steiner complexes in (2) and (3).
Every Steiner complex $\sS_I$ gives rise to a system of contact conics
$\sQ_I=\{\lambda^T Q_I\lambda$, $\lambda\in\P^1(\C)\}$, where $Q_I$ is
a symmetric $2 {\times} 2$-matrix as in (\ref{eq:thirdQMR}),
and a rank-$3$ Gram matrix $G_I$ for $f$.
The following proposition is a direct consequence of Remark~\ref{rem:PSDConic2PSDGram}.

\begin{Prop}\label{prop:GramConic}
 Let $\sS_I$ be a real Steiner complex. The Gram matrix $G_I$ is
 positive semidefinite if and only if the system $\sQ_I$ does not contain any
 real conics.
\end{Prop}

It follows that if $\sS_I$ is one of the three Steiner complexes in (2),
then the Gram matrix $G_I$ is not positive semidefinite,
since the system $\sQ_I$
contains a product of two of the real bitangents
$b_{12},b_{34},b_{56},b_{78}$. Thus  it remains to show that if
 $I = ij$ with $ij \in \{12,34,56,78\}$  as in (3), 
 then the system $\sQ_{ij}$ contains a real conic. 
 
 The symmetric linear determinantal
representation $M$ gives rise to the system $\{\lambda^TM^{\rm
  adj}\lambda\, | \, \lambda\in\P^3(\C)\}$ of (azygetic) contact cubics (see \cite[\S6.3]{Dol}). 
The main idea of the following is that multiplying a bitangent with a 
contact conic of $f$ gives a contact cubic, and if both the bitangent 
and the cubic are real, then the conic must be real. The next lemma 
identifies products of bitangents and contact conics
inside the system of contact cubics  given by $M$.

\begin{Lemma}\label{lem:FromConics2Cubics}
For $i\neq j$ we have $\, b_{ij}\cdot\sQ_{ij} \,= \,\bigl\{\lambda^TM^{\rm
  adj}\lambda\, | \,\lambda\in{\rm span}\{O_i,O_j\}^\perp\bigr\}$.
  \end{Lemma}

\begin{proof}
After a change of coordinates, we can assume that $O_i,O_j,O_k,O_l$ are the four unit vectors $e_1,e_2,e_3,e_4$. This means that $M = xA+yB+zC$ takes the form
\[ M \,\, = \,\,\,
\begin{pmatrix}
  0 & b_{ij} & b_{ik} & b_{il}\\
  b_{ij} & 0 & b_{jk} & b_{jl}\\
  b_{ik} & b_{jk} & 0 & b_{kl}\\
  b_{il} & b_{jl} & b_{kl} & 0
\end{pmatrix} \! .
\]
Consider the three $3 {\times} 3$-minors complementary to the
lower $2 {\times} 2$-block of $M$. They are $\,e_3^TM^{\rm adj}e_3$, $e_3^TM^{\rm adj}e_4$,
$e_4^TM^{\rm adj}e_4$. We check that all three are divisible by $b_{ij}$. Therefore
$b_{ij}^{-1}\cdot\lambda^TM^{\rm adj}\lambda$ with
$\lambda\in{\rm span}\{e_3,e_4\}$ is a system of contact conics. Note that
$b_{ik}b_{jk}=b_{ij}^{-1}e_4^TM^{\rm adj}e_4$. Similarly, we can find
the other six products of pairs of bitangents from
the Steiner complex $\sS_{ij}$, as illustrated by the following picture: 
\vspace{-.6cm}
 \begin{center}      \begin{tikzpicture}
       \fill (1.5,1) circle (2pt) node[above]{$i$}; \fill (3.5,1)
       circle (2pt) node[above]{$j$}; \fill
       (0,0) circle (2pt); \fill (1,0) circle (2pt); \fill (2,0)
       circle (2pt); \fill (3,0) circle (2pt); \fill (4,0) circle
       (2pt); \fill (5,0) circle (2pt); \draw (1.5,1) -- (0,0); \draw
       (3.5,1) -- (0,0); \draw (1.5,1) -- (1,0); \draw (3.5,1) --
       (1,0); \draw[line width=1.5pt] (1.5,1) -- (2,0); \draw[line width=1.5pt]
       (3.5,1) -- (2,0); \draw (1.5,1) -- (3,0); \draw (3.5,1) --
       (3,0); \draw (1.5,1) -- (4,0); \draw (3.5,1) -- (4,0); \draw
       (1.5,1) -- (5,0); \draw (3.5,1) -- (5,0);\draw[line width=1.5pt]
       (3.5,1) -- (1.5,1);
     \end{tikzpicture}
     \end{center}
Hence 
the system of contact conics $\sQ_{ij}$ arises from division by $b_{ij}$ as asserted.
\end{proof}

\begin{proof}[Proof of Theorem~\ref{Thm:TwoK4sIndices} (and hence of
Theorem~\ref{thm:TwoK4s})]
With all the various lemmas in place, only one tiny step is left to be done.
Fix any of the four Steiner complexes $ij$ of type $\bigvee$ in (3).
Then the bitangent $b_{ij}$ is real.  
Since $M$ is real and $\ol O_i=O_j$,
we can pick a real point $\lambda\in{\rm span}\{O_i,O_j\}^{\perp}$.
Lemma \ref{lem:FromConics2Cubics} implies that
 that $\sQ_{ij}$ contains the real conic
$\,b_{ij}^{-1}\cdot \lambda^TM^{\rm adj}\lambda$. Proposition~\ref{prop:GramConic}
now completes the proof.
\end{proof}

\smallskip

Semidefinite programming over the Gram spectrahedron ${\rm Gram}(f)$
means finding the best sum of squares representation of a positive
quartic $f$, where ``best'' refers to some criterion that can be
expressed as a linear functional on Gram matrices.
This optimization problem is of particular interest from the
perspective of Tables 1 and 2 in  \cite{MR2546336}, 
because $m=n=6$ is the smallest instance where
the Pataki range of optimal ranks has size three.
For the definition of {\em Pataki range} see also equation (5.16) in \cite[\S 5]{RS}.
The matrix rank of the exposed vertices of a {\em generic} $6$-dimensional spectrahedron 
 of $6 {\times} 6$-matrices can be either $3$, $4$ or $5$.

The Gram spectrahedra ${\rm Gram}(f)$ are not generic but they exhibit
the generic behavior as far as the Pataki range is concerned.
Namely, if we optimize
a linear function over ${\rm Gram}(f)$
then the rank of the optimal matrix
can be either $3$, $4$ or $5$.
We obtained the following numerical result for the 
distribution of these ranks by optimizing a random linear function 
over ${\rm Gram}(f)$ for randomly chosen $f$:

\begin{table}[htb] 
\centering
\begin{tabular}{|l|c|c|c||c|}
   \hline
   Rank of optimal matrix & $3$ & $4$ & $5$ & any\\
   \hline\hline
   Algebraic degree & 63 & 38 & 1 & 102\\
   \hline
   Probability & 2.01\% & 95.44\% & 2.55\% & 100\%\\
   \hline
 \end{tabular}
\medskip
\caption{Statistics for semidefinite programming over Gram spectrahedra.}
\label{table:RankTable}
\end{table}

The sampling in Table \ref{table:RankTable} was done in {\tt
  matlab}\footnote{\tt www.mathworks.com}, using the random matrix
generator.  
This distribution for the three
possible ranks appears to be close to that of
the generic case, as given in \cite[Table 1]{MR2546336}. 
The algebraic
degree of the optimal solution, however, is much lower than in the
generic situation of \cite[Table 2]{MR2546336}, where the three
degrees are $112$, $1400$ and $32$.  For example, while the rank-$3$
locus on the generic spectrahedron has $112$ points over $\C$, our
Gram spectrahedron ${\rm Gram}(f)$ has only $63$, one for each Steiner
complex.

The greatest surprise in Table \ref{table:RankTable} is the number $1$
for the  algebraic degree of the rank-$5$ solutions.
This means that the optimal solution of
a rational linear function over the Gram spectrahedron ${\rm Gram}(f)$ is
$\Q$-rational whenever it has rank~$5$.
For a concrete example, consider the problem of maximizing
the function
\[
159\lambda_1-9\lambda_2+34\lambda_3+73\lambda_4+105\lambda_5+86\lambda_6
\]
over the Gram spectrahedron ${\rm Gram}(f)$ of
the {\em Fermat quartic} $\, f = x^4+y^4+z^4$.
The optimal solution for this instance is the
rank-$5$ Gram matrix (\ref{eq:GramMatrix})  
with coordinates
\begin{tiny}
\[\lambda = \left(\frac{-867799528369}{6890409751681},
\frac{-7785115393679}{13780819503362},
\frac{-2624916076477}{6890409751681},
 \frac{1018287438360}{6890409751681},
\frac{2368982554265}{6890409751681},\frac{562671279961}{6890409751681}\right).\]
\end{tiny}
The drop from $1400$ to $38$
for the algebraic degree of optimal Gram matrices of rank $4$ is dramatic.
It would be nice to understand the geometry behind this.
We finally note that the algebraic degrees $63, 38, 1$ in
Table~\ref{table:RankTable} were computed using
{\tt Macaulay2}\footnote{{\tt www.math.uiuc.edu/Macaulay2}}
by elimination
from the KKT equations, as described in \cite[\S 5]{RS}.

\section{The Variety of Cayley Octads} \label{sec:octad2}

The Cayley octads form a subvariety of codimension three in the 
space of eight labeled points in $\P^3$. A geometric study of this
variety was undertaken by Dolgachev and Ortland in \cite[\S IX.3]{DO},
building on classical work of  Coble \cite{Cob}.
This section complements their presentation with several
explicit formulas we found useful for constructing
examples and for performing symbolic computations.
Besides convex algebraic geometry \cite{MR2292953,Hen,RS},
our results have potential applications in {\em number theory}
(e.g.~arithmetic of del Pezzo varieties \cite[\S V]{DO})
and {\em integrable systems} (e.g.~$3$-phase solutions
to the Kadomtsev-Petviashvili equation~\cite{DFS}).
In Theorem \ref{thm:VHdisc} we compute the discriminant of
the quartics (\ref{intro:LMR}) and (\ref{intro:SOS}), and in
Proposition \ref{thm:nets}
we discuss an application to nets of real quadrics in $\P^3$.

We begin with the fact that a Cayley octad
is determined by any seven of its points.
Here is a rational formula for the eighth point in terms 
of the first seven.

\begin{Prop} \label{prop:octadeqns} 
Consider a general configuration $\mathcal{C}$ of seven points in $\P^3$,
with coordinates 
$(1{:}0{:}0{:}0)$, $(0{:}1{:}0{:}0)$, 
$(0{:}0{:}1{:}0)$, $(0{:}0{:}0{:}1)$, $(1{:}1{:}1{:}1)$,
$(\alpha_6{:} \beta_6{:} \gamma_6{:}\delta_6)$ and
$(\alpha_7{:} \beta_7{:} \gamma_7{:}\delta_7)$.
% the columns of the $4 {\times} 7$-matrix
%$$ \begin{pmatrix}
 %                  \,1 &   0 &   0  &  0  &  1 &   \alpha_6 &    \alpha_7  \\
%                   \,0 &   1  &  0  &  0  &  1  &  \beta_6 &   \beta_7     \\
 %                  \,0  &  0  &  1  &  0  &  1  &  \gamma_6 &   \gamma_7  \\
 %                 \, 0  &  0  &  0  &  1  &  1  &  \delta_6 & \delta_ 7  \\
%                   \end{pmatrix}. $$
The unique point $(\alpha_8{:} \beta_8{:} \gamma_8{:}\delta_8)$ in $\P^3$
which completes $\mathcal{C}$ to a Cayley octad is given by the
following rational functions in the eight free parameters:
\begin{align*}
\alpha_8 \,\,&=\,\,\frac{\beta_6 \gamma_7-\beta_6 \delta_7-\gamma_6 \beta_7+\gamma_6 \delta_7+\delta_6 \beta_7-\delta_6 \gamma_7}{\beta_6 \gamma_6 \beta_7 \delta_7-\beta_6 \gamma_6 \gamma_7 \delta_7-\beta_6 \delta_6 \beta_7 \gamma_7+\beta_6 \delta_6 \gamma_7 \delta_7+\gamma_6 \delta_6 \beta_7 \gamma_7-\gamma_6 \delta_6 \beta_7 \delta_7} ,\\
\beta_8 \,\,&=\,\,  \frac{\alpha_6 \gamma_7-\alpha_6 \delta_7-\gamma_6 \alpha_7+\gamma_6 \delta_7+\delta_6 \alpha_7-\delta_6 \gamma_7}{\alpha_6 \gamma_6 \alpha_7 \delta_7-\alpha_6 \gamma_6 \gamma_7 \delta_7-\alpha_6 \delta_6 \alpha_7 \gamma_7+\alpha_6 \delta_6 \gamma_7 \delta_7+\gamma_6 \delta_6 \alpha_7 \gamma_7-\gamma_6 \delta_6 \alpha_7 \delta_7} ,\\
\gamma_8 \,\,&=\,\,  \frac{\alpha_6 \beta_7-\alpha_6 \delta_7-\beta_6 \alpha_7+\beta_6 \delta_7+\delta_6 \alpha_7-\delta_6 \beta_7}{\alpha_6 \beta_6 \alpha_7 \delta_7-\alpha_6 \beta_6 \beta_7 \delta_7-\alpha_6 \delta_6 \alpha_7 \beta_7+\alpha_6 \delta_6 \beta_7 \delta_7+\beta_6 \delta_6 \alpha_7 \beta_7-\beta_6 \delta_6 \alpha_7 \delta_7} ,\\
\delta_8 \,\,&=\,\,\frac{\alpha_6 \beta_7-\alpha_6 \gamma_7-\beta_6 \alpha_7+\beta_6 \gamma_7+\gamma_6 \alpha_7-\gamma_6 \beta_7}{\alpha_6 \beta_6 \alpha_7 \gamma_7-\alpha_6 \beta_6 \beta_7 \gamma_7-\alpha_6 \gamma_6 \alpha_7 \beta_7+\alpha_6 \gamma_6 \beta_7 \gamma_7+\beta_6 \gamma_6 \alpha_7 \beta_7-\beta_6 \gamma_6 \alpha_7 \gamma_7} .
\end{align*}
\end{Prop}

\begin{proof}
This can be verified using linear algebra
over the rational function field $K = \Q(\alpha_6,\beta_6,\gamma_6,\delta_6,
\alpha_7,\beta_7,\gamma_7,\delta_7)$.
We compute three linearly independent quadrics that vanish at the
seven points and check that they also
vanish at $(\alpha_8{:} \beta_8{:} \gamma_8{:}\delta_8)$.
\end{proof}

The formula in Proposition~\ref{prop:octadeqns} 
parametrizes the semialgebraic set of
real quartics that consist of four ovals. Indeed,
if the parameters $\,\alpha_6,\beta_6,\gamma_6,\delta_6,
\alpha_7,\beta_7,\gamma_7,\delta_7\,$ are real numbers then the
corresponding quartic curve (\ref{intro:LMR}) 
has $28$ real bitangents, so it falls into the first row of
Table \ref{table:sixtypes}, and all quartics with four ovals arise.
Note that this row is the one relevant for applications to
periodic water waves \cite{DFS}.
In practice we usually choose rational numbers for the parameters.
This represents all curves whose $28$ bitangents are rational,
such as the Edge quartic (\ref{LMRforedge}), and it ensures that 
the ground field is $K= \Q$ for all computations in Sections 3-6.

Yet the above formula has two disadvantages.
First of all, it breaks
the symmetry among the eight points in the Cayley octad,
and, secondly, it does not offer an arithmetically useful parametrization 
for the last two rows of Table \ref{table:sixtypes}. 
Indeed, Vinnikov quartics and positive quartics are the lead actors in
this paper, and we found ourselves unable
to manipulate them properly using Proposition~\ref{prop:octadeqns}.
For example, for a long time we failed to find a quartic
with eight rank-$3$ Gram matrices over $\Q$. Then we derived
Proposition~\ref{prop:octadeqns2}, and this led us to
 Example~\ref{ex:BerndQuartic}.

Let $O$ be a configuration of eight points in general position in $\P^3$,  represented
by a $4 {\times} 8$-matrix. If $O^*$ is another
such matrix whose row space equals the kernel
of $O$ then the configuration represented by $O^*$
is said to be {\em Gale dual} or {\em associated} to $O$.
We refer to \cite{DO, EP} for the basics on Gale duality in the
context of algebraic geometry. Both configurations $O$
and $O^*$ are understood as equivalence classes
modulo projective transformations
of $\P^3$ and relabeling of the eight points. We say that  the configuration $O$ is {\em Gale self-dual}
if %there exists a labeling of the points such that 
$O$ and $O^*$ are equivalent in this sense. By a classical result due to Coble \cite{Cob}, $O$ is Gale self-dual 
if and only if $O$ is a Cayley octad; see Theorem~\ref{thm:coble}.

The variety of Cayley octads is defined by the equation $O = O^*$.
We now translate this equation into an algebraic form that is useful for
computations. Let $p_{ijkl}$ denote the $4 {\times} 4$-minor of
the $4 {\times} 8$-matrix $O$ that represents our configuration
of eight points in $\P^3$. Consider the condition that $O$ is mapped to a configuration
 projectively equivalent to its Gale dual $O^*$ if we relabel
 the points by the permutation $(18)(27)(36)(45)$. We express this condition
 using the {\em Pl\"ucker coordinates}~$p_{ijkl}$.
 
\begin{Prop} \label{prop:octadeqns2}
Eight points in $\P^3$ form a Cayley octad
if and only if
\begin{small}$$\begin{matrix}
p_{1234} p_{1256} p_{3578} p_{4678} =p_{5678} p_{3478} p_{1246} p_{1235} , & 
p_{1234} p_{1257} p_{3568} p_{4678} =p_{5678} p_{3468} p_{1247} p_{1235} , \\
p_{1234} p_{1267} p_{3568} p_{4578} =p_{5678} p_{3458} p_{1247} p_{1236}  , & 
p_{1234} p_{1356} p_{2578} p_{4678} =p_{5678} p_{2478} p_{1346} p_{1235} , \\
p_{1234} p_{1457} p_{2568} p_{3678} =p_{5678} p_{2368} p_{1347} p_{1245} , & 
p_{1234} p_{1467} p_{2568} p_{3578} =p_{5678} p_{2358} p_{1347} p_{1246} , \\
p_{1235} p_{1267} p_{3468} p_{4578} =p_{4678} p_{3458} p_{1257} p_{1236} , & 
p_{1235} p_{1347} p_{2468} p_{5678} =p_{4678} p_{2568} p_{1357} p_{1234} , \\
p_{1235} p_{1367} p_{2468} p_{4578} =p_{4678} p_{2458} p_{1357} p_{1236} , & 
p_{1235} p_{1467} p_{2468} p_{3578} =p_{4678} p_{2358} p_{1357} p_{1246} , \\
p_{1236} p_{1347} p_{2458} p_{5678} =p_{4578} p_{2568} p_{1367} p_{1234} , & 
p_{1236} p_{1456} p_{2478} p_{3578} =p_{4578} p_{2378} p_{1356} p_{1246} , \\
p_{1245} p_{1267} p_{3468} p_{3578} =p_{3678} p_{3458} p_{1257} p_{1246} , & 
p_{1245} p_{1346} p_{2378} p_{5678} =p_{3678} p_{2578} p_{1456} p_{1234} , \\
p_{1245} p_{1356} p_{2378} p_{4678} =p_{3678} p_{2478} p_{1456} p_{1235} , & 
p_{1245} p_{1357} p_{2368} p_{4678} =p_{3678} p_{2468} p_{1457} p_{1235} , \\
p_{1245} p_{1367} p_{2368} p_{4578} =p_{3678} p_{2458} p_{1457} p_{1236} , & 
p_{1246} p_{1357} p_{2368} p_{4578} =p_{3578} p_{2468} p_{1457} p_{1236} , \\
p_{1246} p_{1357} p_{2458} p_{3678} =p_{3578} p_{2468} p_{1367} p_{1245} , & 
p_{1247} p_{1357} p_{2368} p_{4568} =p_{3568} p_{2468} p_{1457} p_{1237} , \\
\qquad \qquad \qquad \qquad \qquad
\text{and} &
p_{1346} p_{1357} p_{2458} p_{2678} =p_{2578} p_{2468} p_{1367} p_{1345}.
\end{matrix}
$$\end{small}\end{Prop}

Before discussing the proof of this theorem, we first explain why the shape
of the above equations is plausible. Consider the condition for
six points $(x_i:y_i:z_i)$ in $\P^2$ to be self-dual, in the sense above.
This condition means that the six points lie on a conic, and we write this
algebraically in terms of Pl\"ucker coordinates as
\begin{equation} \label{eq:pascal}
\,\, {\rm det} \!
\begin{smaller}
\begin{pmatrix}
x_1^2 &  y_1^2 & z_1^2 & x_1 y_1 & x_1 z_1 & y_1 z_1 \\
x_2^2 &  y_2^2 & z_2^2 & x_2 y_2 & x_2 z_2 & y_2 z_2 \\
x_3^2 &  y_3^2 & z_3^2 & x_3 y_3 & x_3 z_3 & y_3 z_3 \\
x_4^2 &  y_4^2 & z_4^2 & x_4 y_4 & x_4 z_4 & y_4 z_4 \\
x_5^2 &  y_5^2 & z_5^2 & x_5 y_5 & x_5 z_5 & y_5 z_5 \\
x_6^2 &  y_6^2 & z_6^2 & x_6 y_6 & x_6 z_6 & y_6 z_6 
\end{pmatrix} 
\end{smaller}
\,\, = \,\,\,
p_{123} p_{145} p_{246} p_{356} -
 p_{124} p_{135} p_{236} p_{456} . 
 \end{equation}
This formula appears in \cite[Ex.4, p.37]{DO}
and we adapt the derivation given there.

\begin{proof}[Sketch of proof for Proposition \ref{prop:octadeqns2}]
The cross ratio $\,(p_{1234} p_{1256})/(p_{1235} p_{1246})\,$
is invariant under projective transformations. 
The permutation $ (18)(27)(36)(45)$ of the points transforms that
cross ratio into $\,(p_{8765} p_{8743})/(p_{8764} p_{8753})$.
The condition $O =O^*$ implies that these two cross ratios
are equal. By clearing denominators, the equality of cross ratios
translates into the first of the $21$ equations listed above:
$$ p_{1234} p_{1256} p_{3578} p_{4678} \,\,= \,\, p_{5678} p_{3478} p_{1246} p_{1235} $$
The other $20$ equations are found by the same argument for
cross ratios. By incorporating the quadratic Pl\"ucker relations among the $p_{ijkl}$,
we check that our list of $21$ cross ratio identities is complete, in the sense
that it ensures $\,O = O^*$.
\end{proof}

Dolgachev and Ortland \cite[page 176]{DO} present the conditions under which
a regular Cayley octad $O$ can degenerate. Their analysis exhibits
$64 = 28+35+1$ boundary divisors in the compactified space of
Cayley octads. These are as follows:
\begin{enumerate}
\item Two points of $O$ can come together. This gives $28 = \binom{8}{2}$
boundary divisors, e.g., points $1$ and $2$ come together
if and only if $\,p_{12ij} = 0$ for $3 {\leq} i  {<} j {\leq} 8$.
\item Four points of $O$ can become coplanar. 
The equations in Proposition \ref{prop:octadeqns2} then ensure
that the other four points become coplanar as well.
 So, in total there are $35 = \frac{1}{2} \binom{8}{4}$
boundary divisors such as  $\,\{p_{1234} = p_{5678} = 0 \}$.
\item The eight points of $O$ can lie on a twisted cubic curve,
which is the intersection of the three quadrics. 
The condition for seven points in $\P^3$ to lie on a twisted cubic curve has codimension $2$.
 H.S.~White \cite[eqn.(2)]{Whi} writes this condition by adding an index to (\ref{eq:pascal}). 
 This gives $7 {\cdot} 15 {\cdot} 3$ equations~like
\begin{equation}
\label{eq:liftedpascal}
p_{1237} p_{1457} p_{2467} p_{3567} -  p_{1247} p_{1357} p_{2367} p_{4567}  \quad = \quad 0 .
\end{equation}
Applying the symmetric group $S_8$ to the indices, 
we obtain equations for the codimension $4$
locus of octads that lie on a twisted cubic curve.
This locus is a divisor in the compactified space of Cayley octads, 
as in \cite[\S IX.3]{DO}. Equivalently,
the equations (\ref{eq:liftedpascal}) imply those in 
Proposition \ref{prop:octadeqns2}.
 \end{enumerate}

\smallskip

We now shift gears and examine the three types of boundary divisors from 
the perspective of the desirable representations (\ref{intro:LMR}) and (\ref{intro:SOS})
of a ternary quartic~$f$. In other words, we wish to identify the 
conditions, expressed algebraically in terms of these two representations, for the
quartic curve $\mathcal{V}_\C(f)$ to become singular. 

Recall that the discriminant
$\Delta$ of $f$ is a homogeneous polynomial of degree $27$,
featured explicitly in \cite[Proposition 6.5]{SSS},
in the $15$ coefficients $c_{ijk}$ of (\ref{intro:quartic}).
If we take $f$ in the representation (\ref{intro:LMR}) then each coefficient $c_{ijk}$ 
is replaced by a polynomial of degree $4$ in the $30=10+10+10$ entries of the symmetric
matrices $A,B$ and $C$. The result of performing this substitution
in the discriminant $\Delta(c_{ijk})$ is  denoted $\Delta(A,B,C)$. This is a 
homogeneous polynomial of degree $108$ in $30$ unknowns.
We call $\Delta(A,B,C)$ the {\em Vinnikov discriminant} of a ternary quartic.

Similarly, if we take $f$ in the representation (\ref{intro:SOS}) then each
coefficient $c_{ijk}$ is replaced by a polynomial of degree $2$ in the $18=6+6+6$
coefficients of the quadrics $q_1, q_2$ and $q_3$.
The result of performing this substitution
in the discriminant $\Delta(c_{ijk})$ is  denoted $\Delta(q_1,q_2,q_3)$. This is a 
homogeneous polynomial of degree $54$ in $18$ unknowns.
We call $\Delta(q_1,q_2,q_3)$ the {\em Hilbert discriminant} of a ternary quartic.

\begin{Thm} \label{thm:VHdisc}
The irreducible factorization of the Vinnikov discriminant equals
\begin{equation}
\label{eq:vinnikovdisc}
 \Delta(A,B,C) \,\, = \,\, {\bf M}(A,B,C) \cdot {\bf P}(A,B,C)^2, 
\end{equation}
where ${\bf P}$ has degree $30$ and corresponds to the boundary divisor (2),
while  ${\bf M}$ has degree $48$, and this is the {\em mixed discriminant} corresponding to
both (1) and (3). \\
The irreducible factorization of the Hilbert discriminant equals
\begin{equation}
\label{eq:hilbertdisc}
 \Delta(q_1,q_2,q_3) \,\, = \,\, {\bf Q}(q_1,q_2,q_3) \cdot {\bf R}(q_1,q_2,q_3)^2, 
 \end{equation}
where ${\bf Q}$ has degree $30$ and the degree $12$ factor ${\bf R}$
is the resultant of $q_1, q_2$ and $q_3$.
\end{Thm}

This theorem is proved by a computation, the details of which we omit
here.  It has been pointed out to us by Igor Dolgachev and Giorgio
Ottaviani that the factorization (\ref{eq:vinnikovdisc}) was already known to
Salmon \cite{MR0115124}, who refers to ${\bf M}(A,B,C)$ as the {\em tact
invariant}. See also \cite[Section 10]{Gi} for a modern treatment.

We discuss the geometric meaning of the factors in
(\ref{eq:vinnikovdisc}) and (\ref{eq:hilbertdisc}).  The polynomials
${\bf M}, {\bf P}$, ${\bf Q}$, and ${\bf R}$ are absolutely
irreducible: they do not factor over $\C$.  The polynomial ${\bf P}$
represents the condition that the span of $A$, $B$ and $C$ in the
space of $4 {\times} 4$-symmetric matrices contains a rank-$2$
matrix. Note that the variety of such rank-$2$ matrices has
codimension $3$ and degree $10$.  The {\em Chow form} of that variety
is precisely our polynomial ${\bf P}$, which explains why ${\bf P}$
has degree $3 \cdot 10 = 30$.

Non-vanishing of  the mixed discriminant ${\bf M}$ is the condition for the 
intersection of three quadrics in $\P^3$ to  be zero-dimensional and smooth.
A general formula for the degree of such discriminants appears in
\cite[Theorem 3.1]{Nie}. It implies that ${\bf M}$ is tri-homogeneous of
degree $(16,16,16)$ in the entries of $(A,B,C)$, so 
the total degree of $M$ is $ 48$. Note that vanishing of
${\bf M}$ represents not just condition (1)
but it also subsumes condition (3)  that the quadrics
 intersect in a twisted cubic curve.

The resultant ${\bf R}$ of three ternary quadrics $(q_1,q_2,q_3)$ 
is tri-homogeneous of degree $(4,4,4)$ since two
quadrics meet in $4$ points in $\P^2$. Thus ${\bf R}$ has total degree
$12$. The extraneous factor ${\bf Q}$ of degree $30$
expresses the condition that, at some point in $\P^2$,
the vector $ (q_1,q_2,q_3)$ is non-zero and lies in the
kernel of its Jacobian.

\smallskip

We close this paper by reinterpreting Table \ref{table:sixtypes}
as a tool to study linear spaces of symmetric $4 {\times} 4$ matrices.
Two matrices $A$ and $B$ determine a {\em pencil of quadrics} in $\P^3$, and three matrices
$A, B, C$ determine a {\em net of quadrics} in $\P^3$.  We now consider 
these pencils and nets over the field $\R$ of real numbers.
A classical fact, proved by Calabi in \cite{Cal}, states that a pencil of quadrics either 
has a common point or contains a positive definite quadric.
This fact is the foundation for an optimization technique
 known in engineering as the {\em S-procedure}.
The same dichotomy is false for nets of quadrics \cite[\S 4]{Cal},
and for quadrics in $\P^3$ it fails in two interesting ways.

\begin{Thm} \label{thm:nets}
Let $\mathcal{N}$ be a real net of homogeneous quadrics in four unknowns with $\Delta(\mathcal{N})\neq 0$. 
Then  precisely one of the following four cases holds:
\begin{itemize}
\item[(a)] The quadrics in $\mathcal{N}$ have a common point in $\P^3(\R)$.
\item[(b)] The net $\mathcal{N}$ is definite, \textit{i.e.}~it contains a positive definite quadric.
%\item[(c)] The polynomial $\det(\mathcal{N})$ is hyperbolic with respect to a nondefinite $N\in \mathcal{N}$. 
\item[(c)] There is a definite net $\mathcal{N'}$ with $\det(\mathcal{N'})=\det(\mathcal{N})$, but $\mathcal{N}$ is nondefinite.
%\item[(c)] The polynomial $\det(\mathcal{N})$ is hyperbolic with respect to a nondefinite $N\in \mathcal{N}$. Equivalently, there exists a net $\mathcal{N'}$ satisfying (b) with $\det(\mathcal{N'})=\det(\mathcal{N})$, but $\mathcal{N}$ itself contains no definite quadric.
\item[(d)] The net $\mathcal{N}$ contains no singular quadric.
%\item[(d)] The $4 \times 4$-determinant restricted to $\mathcal{N}$ is a sum of squares.
\end{itemize}
\end{Thm}

\begin{proof}
  For a real net of quadrics, $\mathcal{N} = \R\{A,B,C\}$, the Vinnikov
  discriminant $\Delta(A,B,C)$ in (\ref{eq:vinnikovdisc}) is
  independent (up to scaling) of the basis $\{A,B,C\}$, and thus can be denoted
  $\Delta(\mathcal{N})$. If $\Delta(\mathcal{N})$ is non-zero, 
  the polynomial $\det(\mathcal{N})=\det(xA+yB+zC)$ defines a smooth curve, which depends on the choice of basis $\{A,B,C\}$ only 
  up to projective change of coordinates in $[x:y:z]$.
  This real quartic falls into precisely one of the six
  classes in Table \ref{table:sixtypes}.  The first four classes
  correspond to our case (a).  The fifth class corresponds to our cases
  (b) and (c) by the Helton-Vinnikov Theorem \cite{MR2292953}. 
 As a Vinnikov quartic has definite and non-definite real determinantal representations, both (b) and (c) do occur \cite{MR1225986}. 
 For an example, see \cite[Ex. 5.2]{us}.
The last class corresponds to our case (d).
\end{proof}

Given a net of quadrics $\mathcal{N} = \R\{A,B,C\}$, one may wish to
know whether there is a common intersection point in real projective
$3$-space $\P^3(\R)$, and, if not, one seeks the certificates promised
in parts (b)--(d) of Theorem \ref{thm:nets}. Our algorithms in
Sections 3, 4 and 5 furnish a practical method for identifying cases
(b) and (d). The difference between (b) and (c) is more subtle and is
discussed in detail in \cite{us}.

\medskip

{\bf Acknowledgments.} We are grateful to
Morgan Brown and Mauricio Velasco for
their help at the start of this project, and we thank
Dustin Cartwright, Melody Chan,
Igor Dolgachev,
 Giorgio Ottaviani  and Bruce Reznick
 for their comments on the manuscript.
Daniel Plaumann was supported by the Alexander-von-Humboldt 
Foundation through a Feodor Lynen postdoctoral fellowship,
hosted at UC Berkeley.
Bernd Sturmfels and Cynthia Vinzant acknowledge  financial support by
 the U.S.~National Science Foundation
(DMS-0757207 and DMS-0968882).

\nocite{049.1315cj}
{\small\linespread{1}

}
\end{document}